\documentclass[a4paper,11pt,reqno]{amsart}
\usepackage{amsmath,amsfonts,amsthm,amssymb,color}
\usepackage[T1]{fontenc}
\usepackage{pdfsync}
\usepackage{csquotes}
\usepackage{graphicx}
\usepackage{pstricks}
\usepackage{lmodern}
\usepackage[latin1]{inputenc}

\newcommand{\be}{\beta}

\newcommand{\id}{\mbox{Id}}

\newcommand{\1}{{\bf 1}}

\newcommand{\Dti}{\tilde{D}}
\newcommand{\tha}{\hat{t}}

\newcommand{\Xha}{\hat{X}}
\newcommand{\Xti}{\tilde{X}}

\newcommand{\cq}{\mathcal{Q}}
\newcommand{\bu}{\mathbf{U}}
\newcommand{\bv}{\mathbf{V}}
\newcommand{\by}{\mathbf{Y}}

\newcommand{\C}{\mathbb C}

\newcommand{\R}{\mathbb R}
\newcommand{\N}{\mathbb N}

\newcommand{\ca}{\mathcal A}
\newcommand{\cb}{\mathcal B}
\newcommand{\cac}{\mathcal C}

\newcommand{\cl}{\mathcal L}

\newcommand{\cn}{\mathcal N}

\newcommand{\cs}{\mathcal S}

\newcommand{\al}{\alpha}
\newcommand{\der}{\delta}

\newcommand{\ga}{\gamma}

\newcommand{\la}{\lambda}

\newcommand{\vp}{\varphi}

\newcommand{\lln}{\left|}
\newcommand{\rrn}{\right|}

\newtheorem{theorem}{Theorem}[section]

\newtheorem{corollary}[theorem]{Corollary}

\newtheorem{definition}[theorem]{Definition}

\newtheorem{lemma}[theorem]{Lemma}

\newtheorem{proposition}[theorem]{Proposition}

\theoremstyle{remark}
\newtheorem{remark}[theorem]{Remark}
\newtheorem{example}[theorem]{Example}

\date{\today}

\begin{document}

\makeatletter
\def\@settitle{\begin{center}%
  \baselineskip14\p@\relax
    \normalfont\LARGE
\@title
  \end{center}%
}
\makeatother

\title{On stochastic calculus with respect to $q$-Brownian motion}

\author{Aur\'elien Deya}
\address[A. Deya]{Institut Elie Cartan, University of Lorraine
B.P. 239, 54506 Vandoeuvre-l\`es-Nancy, Cedex
France}
\email{aurelien.deya@univ-lorraine.fr}

\author{Ren\'e Schott}
\address[R. Schott]{Institut Elie Cartan, University of Lorraine
B.P. 239, 54506 Vandoeuvre-l\`es-Nancy, Cedex
France}
\email{rene.schott@univ-lorraine.fr}

\keywords{non-commutative stochastic calculus; q-Brownian motion; rough paths theory}

\subjclass[2010]{46L53,60H05,60F17}

\begin{abstract}
Following the approach and the terminology introduced in [A. Deya and R. Schott, On the rough paths approach to non-commutative stochastic calculus, J. Funct. Anal., 2013], we construct a product L{\'e}vy area above the $q$-Brownian motion (for $q\in [0,1)$) and use this object to study differential equations driven by the process.\\
\indent We also provide a detailled comparison between the resulting \enquote{rough} integral and the stochastic \enquote{It{\^o}} integral exhibited by Donati-Martin in [C. Donati-Martin, Stochastic integration with respect to $q$ Brownian motion, Probab. Theory Related Fields, 2003].
\end{abstract} 

\maketitle

\section{Introduction: the $q$-Brownian motion}\label{sec:intro}

The $q$-Gaussian processes (for $q\in [0,1)$) stand for one of the most standard families of non-commutative random variables in the literature. Their consideration can be traced back to a paper by Frisch and Bourret in the early 1970s \cite{FB}: the dynamics is therein suggested as a model to quantify some possible {non-commutativity phenomenon} between the creation and annihilator operators on the Fock space, the limit case $q= 1$ morally corresponding to the classical probability framework. The mathematical construction and basic stochastic properties of the $q$-Gaussian processes were then investigated in the 1990s, in a series of pathbreaking papers by Bo{\.z}ejko, Kümmerer and Speicher \cite{q-gauss,BSp1,BSp2}.

\smallskip

For the sake of clarity, let us briefly recall the framework of this analysis and introduce a few notations that will be used in the sequel (we refer the reader to the comprehensive survey \cite{nicaspeicher} for more details on the subsequent definitions and assertions). First, recall that the processes under consideration consist of paths with values in a non-commutative probability space, that is a von Neumann algebra $\mathcal{A}$ equipped with a weakly continuous, positive and faithful trace $\vp$ . The sole existence of such a trace $\vp$ on $\ca$ (to be compared with the \enquote{expectation} in this setting) is known to give the algebra a specific structure, with \enquote{$L^p$}-norms
$$\|X\|_{L^p(\vp)}:= \vp( |X|^p)^{1/p} \quad  \ (\, |X|:=\sqrt{XX^\ast}\, )$$
closely related to the operator norm $\|.\|$: 
\begin{equation}\label{prop:norm-op}
\|X\|_{L^p(\vp)} \leq \|X\| \quad , \quad \|X\|=\lim_{p\to \infty} \|X\|_{L^p(\vp)} \ , \ \text{for all} \ X\in \ca \ .
\end{equation}
Now recall that non-commutative probability theory is built upon the following fundamental spectral result: any element $X$ in the subset $\ca_\ast$ of self-adjoint operators in $\ca$ can be associated with a law that shares the same moments. To be more specific, there exists a unique compactly supported
probability measure $\mu$ on $\R$ such that for any real polynomial $P$,
\begin{equation}\label{mu}
\int_{\R} P(x) \mathrm{d}\mu(x) = \vp(P(X))\ .
\end{equation}
Based on this property, elements in $\ca_\ast$ are usually referred to as (non-commutative) random variables, and in the same vein, the law of a given family $\{X^{(i)}\}_{i\in I}$ of random variables in $(\ca,\vp)$ is defined as the set of all of its joint moments 
$$\vp\big(X^{(i_1)} \cdots X^{(i_r)}\big) \ , \quad i_1,\ldots,i_r \in I \ , \ r\in \N \ .$$

\smallskip

With this stochastic approach in mind and using the terminology of \cite{q-gauss}, the definition of a $q$-Gaussian family can be introduced along the following combinatorial description:

\begin{definition}\label{d:crossing}
1. Let $r$ be an even integer. A \emph{pairing} of $\{1,\ldots,r\}$ is any partition of $\{1,\ldots,r\}$ into $r/2$ disjoint subsets, each of cardinality $2$. We denote by $\mathcal{P}_2(\{1,\ldots,r\})$ or $\mathcal{P}_2(r)$ the set of all pairings of $\{1,\ldots,r\}$.

2. When $\pi \in \mathcal{P}_2(\{1,\ldots,r\})$, a \emph{crossing} in $\pi$ is any set of the form $\{\{x_1,y_1\},\{x_2,y_2\}\}$ with $\{x_i,y_i\}\in \pi$ and $x_1 < x_2 <y_1 <y_2$. The number of such crossings is denoted by $\mathrm{Cr}(\pi)$. 
\end{definition}

\begin{definition}
For any fixed $q\in [0,1)$, we call a $q$-Gaussian family in a non-commutative probability space $(\ca,\vp)$ any
collection $\{X_i\}_{i \in I}$ of random variables in $(\ca,\vp)$ such that, for every integer $r\geq $1 and
all $i_1,\ldots,i_r \in I$, one has
\begin{equation}\label{form-q-gaussian}
\vp\big( X_{i_1}\cdots X_{i_r}\big)=\sum_{\pi \in \mathcal{P}_2(\{1,\ldots,r\})} q^{\mathrm{Cr}(\pi)} \prod_{\{p,q\}\in \pi} \vp\big( X_{i_p}X_{i_q}\big) \ .
\end{equation}
\end{definition}

\

Therefore, just as with classical (commutative) Gaussian families, the law of a $q$-Gaussian family $\{X_i\}_{i \in I}$ is completely characterized by the set of its covariances $\vp(X_iX_j)$, $i,j\in I$. In fact, when $q\to 1$ and $\vp$ is - at least morally - identified with the usual expectation, relation (\ref{form-q-gaussian}) is nothing but the classical Wick formula satisfied by the joint moments of Gaussian variables. 

\smallskip

When $q=0$, such a family of random variables is also called a semicircular family, in reference to its marginal distributions (see \cite[Chapter 8]{nicaspeicher} for more details on semicircular families, in connection with the so-called free central limit theorem).   

\smallskip

We are now in a position to introduce the family of processes at the core of our study:
\begin{definition}\label{defi:q-bm}
For any fixed $q\in [0,1)$, we call $q$-Brownian motion ($q$-Bm) in some non-commutative probability space $(\ca,\vp)$ any $q$-Gaussian family $\{X_t\}_{t\geq 0}$ in $(\ca,\vp)$ with covariance function given by the formula
\begin{equation}\label{cova-q-bm}
\vp\big( X_{s}X_{t}\big)=s\wedge t \ .
\end{equation}
\end{definition}

The existence of such a non-commutative process (in some non-commutative space $(\ca,\vp)$) has been established by Bo{\.z}ejko and Speicher in \cite{BSp1}. In the same spirit as above, the $q$-Bm distribution can be regarded as a straightforward extension of two well-known processes:

\smallskip

\noindent
$\bullet$ When $q\to 1$, one recovers the classical Brownian-motion dynamics, with independent, stationary and normally-distributed increments.

\smallskip

\noindent
$\bullet$ The $0$-Brownian motion coincides with the celebrated free Brownian motion, whose freely-independent increments are known to be closely related to the asymptotic behaviour of large random matrices, following Voiculescu's breakthrough results \cite{voiculescu}. 

\smallskip

Thus, we have here at our disposal a family of processes which, as far as distribution is concerned, provides a natural \enquote{smooth} interpolation between two of the most central objects in probability theory: the standard and the free Brownian motions. It is then natural to wonder whether the classical stochastic properties satisfied by each of these two processes can be \enquote{lifted} on the level of this interpolation, or in other words if the properties known for $q=0$ and $q\to 1$ can be extended to every $q\in [0,1)$. Of course, any such extension potentially offers an additional piece of evidence in favor of this interpolation model, as a privileged link between the free and the commutative worlds. 

\smallskip

Some first results in this direction, focusing on the stationarity property and the marginal-distribution issue, can be found in \cite{q-gauss}:

\begin{proposition}\label{prop:base-q-bm}
For any fixed $q\in [0,1)$, let $\{X_t\}_{t\geq 0}$ be a $q$-Brownian motion in some non-commutative probability space $(\ca,\vp)$. 
Then for all $0\leq s<t$, the random variable $X_t-X_s$ has the same law as $\sqrt{t-s}\, X_1$, in the sense of identity (\ref{mu}). In particular, any $q$-Brownian motion $\{X_t\}_{t\geq 0}$ is a $\frac12$-Hölder path in $\ca$, i.e.
\begin{equation}\label{q-bm-holder}
\sup_{s<t} \frac{\|X_t-X_s\|}{\lln t-s \rrn^{1/2}} \leq \|X_1\| \ < \ \infty\ .
\end{equation}
Moreover, the law $\mu_q$ of $X_1$ is absolutely continuous with respect to the Lebesgue measure; its density is supported on $\big[\frac{-2}{\sqrt{1-q}},\frac{2}{\sqrt{1-q}} \big]$ and is given, within this interval, by the formula
\[
\mu_q(\mathrm{d}x)= \frac{1}{\pi} \sqrt{1-q} \sin \theta \prod_{n=1}^\infty (1-q^n) |1-q^n e^{2i\theta}|^2\ , \quad \text{where} \ x=\frac{2\cos \theta}{\sqrt{1-q}} \,\,\mbox{ with } \theta \in [0,\pi]\ .
\]
\end{proposition}

\smallskip

A next natural step is to examine the possible extension, to all $q\in [0,1)$, of the \emph{stochastic integration} results associated with the free/classical Brownian motion. Let us here recall that the foundations of stochastic calculus with respect to the free Brownian motion (that is, for $q=0$) have been laid in a remarkable paper by Biane and Speicher \cite{biane-speicher}. Among other results, the latter study involves the construction of a free Itô integral, as well as an analysis of the free Wiener chaoses generated by the multiple integrals of the free Brownian motion. 

\smallskip

These lines of investigation have been followed by Donati-Martin in \cite{donati} to handle the general $q$-Bm case, with the construction of a $q$-Itô integral and a study of the $q$-Wiener chaos. Let us also mention the results of \cite{DNN} related to the extension of the fourth-moment phenomenon that prevails in Wiener chaoses. 

\

In this paper, we intend to go further with the analysis related to the $q$-Brownian motion. To be more specific, we propose, in the continuation of \cite{deya-schott}, to adapt some of the main \emph{rough-path principles} to this setting. The aim here is to derive a very robust integration theory allowing, in particular, to consider the study of differential equations driven by the $q$-Bm, i.e. sophisticated dynamics of the form
\begin{equation}\label{equa-diff-q-bm}
dY_t=f(Y_t) \cdot dX_t \cdot g(Y_t) \ , 
\end{equation}
for smooth functions $f,g$.

\smallskip

In fact, thanks to the general (non-commutative) rough-path results proved in \cite{deya-schott} (and which we will recall in Section \ref{sec:general-rp}), the objective essentially reduces to the exhibition of a so-called \emph{product Lévy area} above the $q$-Bm, that is a kind of iterated integral of the process involving the product structure of $\ca$. Let us briefly recall that the definition of such an object (which appears as quite natural in this algebra setting) has been introduced in \cite{deya-schott} as a way to overcome the possible non-existence problems arising from the study of more general Lévy areas, in the classical Lyons' sense \cite{lyons-book} (see \cite{vic-free} for a description of the non-existence issue in the free case).

\

At this point, we would like to draw the reader's attention to the fact that the construction in \cite{donati} of a $q$-Itô integral as an element of $L^2(\vp)$ would not be not sufficient for our purpose. Indeed, the rough-path techniques are based on Taylor-expansion procedures, which, for obvious stability reasons, forces us to consider an algebra norm in the computations. As a result, any satisfying notion of product Lévy area requires some control with respect to the operator norm, that is in $L^\infty(\vp)$ (along (\ref{prop:norm-op})), and not only with respect to the $L^2(\vp)$-norm (see Section \ref{sec:general-rp} and especially Definition \ref{defi:aire-gene} for more details on the topology involved in this control).

In the particular case of the free Bm ($q=0$), the Burkholder-Gundy inequality established by Biane and Speicher in \cite[Theorem 3.2.1]{biane-speicher} immediately gives rise to operator-norm controls on the free Itô integral, which we could readily exploit in \cite{deya-schott} to deal with rough paths in the free situation. Unfortunately, and at least for the time being, no similar operator-norm control has been shown for the $q$-Itô integral when $q\in (0,1)$. With our rough-path objectives in mind, we will be able to overcome this difficulty though, by resorting to a straightforward $L^\infty(\vp)$-construction of a product Lévy area - the latter object being actually much more specific than a general Itô integral. This is the purpose of the forthcoming Section \ref{sec:levy-area}, which leads to the main result of the paper. Injecting this construction into the general rough-path theory will immediately answer our original issue, that is the derivation of a robust stochastic calculus for the $q$-Bm.

\

It is then possible to compare, a posteriori, the resulting rough integral with more familiar $q$-Itô or $q$-Stratonovich integrals, through a standard $L^2(\vp)$-analysis and the involvement of the so-called second-quantization operator. This comparison will be the topic of Section \ref{sec:comparison}. Let us however insist, one more time, on the fact that this sole $L^2(\vp)$-analysis would not have been sufficient for the rough-path theory (and the powerful rough-path results) to be applied in this situation.

\

Our construction of a product Lévy area will only rely on the consideration of the law of the $q$-Bm, that is on the process as given by Definition \ref{defi:q-bm}. In other words, no reference will be made to any particular representation of the process as a map with values in some specific algebra (just as classical probability theory builds upon the law of the Brownian motion and not upon its representation). The only reference to some particular representation of the $q$-Brownian motion (namely its standard representation on the $q$-Fock space) will occur in Section \ref{sec:comparison}, as a way to compare our rough objects with the constructions of \cite{donati}, based on the Fock space.

Besides, we have chosen in this study to focus on the case where $q\in[0,1)$ and introduce the $q$-Brownian motion as a natural interpolation between the free and the standard Brownian motions. We are aware that the definition of a $q$-Bm can also be extended to every $q\in (-1,0)$, that is up to the \enquote{anticommutative} situation $q\to -1$. In fact, we must here specify that the positivity assumption on $q$ will be used in an essential way for the construction of the product Lévy area (see for instance (\ref{positivity-assumption})), and at this point, we do not know if such an object could also be exhibited in the case $q<0$.

\

As we already sketched it in the above description of our results, the study is organized as follows. In Section \ref{sec:general-rp}, we will recall the general non-commutative rough-path results obtained in \cite{deya-schott} and at the core of the present analysis. Section \ref{sec:levy-area} is devoted to the construction of the main object involved in the rough-path procedure, that is a product Lévy area above the $q$-Bm. Finally, Section \ref{sec:comparison} focuses on the $L^2(\vp)$-comparison of the rough constructions with more standard Itô/Stratonovich definitions.

\

\noindent
\textbf{Acknowledgements.} We are very grateful to two anonymous referees for their attentive reading and insightful suggestions.

\

\section{General rough-path results in $C^\ast$-algebras}\label{sec:general-rp}

Our strategy to develop a robust $L^\infty(\vp)$-stochastic calculus for the $q$-Bm is based on the non-commutative rough-path considerations of \cite[Section 4]{deya-schott}. Therefore, before we can turn to the $q$-Bm situation, it is necessary for us to recall the main results of the theoretical analysis carried out in \cite{deya-schott}. This requires first a few brief preliminaries on functional calculus in a $C^\ast$-algebra (along the framework of \cite{biane-speicher}), as well as precisions on the topologies involved in the study. Special emphasis will be put on the cornerstone of the rough-path machinery, the \emph{product Lévy area}, around which the whole integration procedure can be naturally expanded. 

Note that the considerations of this section apply to a general $C^\ast$-algebra $\ca$, that we fix from now on. In particular, no additional trace operator will be required here. As before, we denote by $\|.\|$ the operator norm on $\ca$, and $\ca_\ast$ will stand for the set of self-adjoint operators in $\ca$. We also fix an arbitrary time horizon $T>0$ for the whole section.

\subsection{Tensor product}\label{subsec:tensor}

Let $\ca \otimes \ca$ be the algebraic tensor product generated by $\ca$, and just as in \cite{biane-speicher}, denote by $\sharp$ the natural product interaction between $\ca$ and $\ca\otimes \ca$, that is the linear extension of the formula
$$(U_1 \otimes U_2) \sharp X=X\sharp (U_1 \otimes U_2):=U_1X  U_2\ , \quad \text{for all} \ U_1,U_2,X \in \ca \ .$$
In a similar way, set, for all $U_1,U_2,U_3,X \in \ca$,
$$X \sharp (U_1 \otimes U_2 \otimes U_3) :=(U_1X U_2) \otimes U_3 \quad , \quad (U_1 \otimes U_2 \otimes U_3) \sharp X:=U_1 \otimes (U_2XU_3) \ .$$
Our developments will actually involve the \emph{projective tensor product} $\ca \hat{\otimes}\ca$ of $\ca$, that is the completion of $\ca \otimes \ca$ with respect to the norm
$$\|\bu\|=\|\bu\|_{\ca \hat{\otimes} \ca}:=\inf \sum_i \| U_i\| \|V_i\| \ ,$$
where the infimum is taken over all possible representation $\bu=\sum_i U_i \otimes V_i$ of $\bu$. It is readily checked that for all $\mathbf{U}\in \ca \otimes \ca $ and $X\in \ca$, one has $\| \mathbf{U} \sharp X\| \leq \|\mathbf{U}\| \|X\|$, and so the above $\sharp$-product continuously extends to $\ca \hat{\otimes}\ca$.
These considerations can, of course, be generalized to the $n$-th projective tensor product $\ca^{\hat{\otimes}n}$, $n\geq 1$, and we will still denote by $\| .\|$ the projective tensor norm on $\ca^{\hat{\otimes}n}$.

Along the same terminology as in \cite{biane-speicher}, we will call any process with values in $\ca \hat{\otimes} \ca$, resp. $\ca \hat{\otimes} \ca \hat{\otimes}\ca$, a \emph{biprocess}, resp. a triprocess.

\subsection{Functional calculus in a $C^\ast$-algebra}\label{subsec:funct} Following again the presentation of \cite{biane-speicher}, let us introduce the class of functions $f$ defined for every integer $k\geq 0$ by
\begin{equation}
\mathbf{F}_k:=\{f:\R \to \C: \ f(x)=\int_{\R} e^{\imath \xi x} \mu_f(\mathrm{d}\xi) \ \text{with} \ \int_{\R} |\xi|^i \, \mu_f(\mathrm{d}\xi) < \infty \ \text{for every} \ i \in \{0,\ldots,k\}\},
\end{equation}
and set, if $f\in \mathbf{F}_k$, $\|f\|_k:=\sum_{i=0}^k \int_{\R} |\xi|^i \, \mu_f(\mathrm{d}\xi)$. Then, with all $f\in \mathbf{F}_0$ and $X\in \ca_\ast$, we associate the operator $f(X)$ along the formula
$$f(X):=\int_{\R} e^{\imath \xi X} \mu_f(\mathrm{d}\xi) \ ,$$
where the integral in the right-hand side is uniformly convergent in $\ca$. This straightforward operator extension of functional calculus happens to be compatible with Taylor expansions of $f$, a central ingredient towards the application of rough-path techniques. The following notion of tensor derivatives naturally arises in the procedure (see the subsequent Examples \ref{ex-1} and \ref{ex-2}):

\smallskip

\begin{definition}
For every $f\in \mathbf{F}_1$, resp. $f\in \mathbf{F}_2$, we define the \emph{tensor derivative}, resp. \emph{second tensor derivative}, of $f$ by the formula: for every $X\in \ca_\ast$,
$$\partial f(X):=\int_0^1 \mathrm{d}\al \int_{\R} \imath \xi \, [e^{\imath \al \xi X} \otimes  e^{\imath (1-\al)\xi X}]\, \mu_f(\mathrm{d}\xi) \quad \in \ca \hat{\otimes} \ca\ ,$$
$$\text{resp.} \quad \partial^2 f(X):=-\iint_{\substack{\al,\be \geq 0\\ \al+\be \leq 1}} \mathrm{d}\al \, \mathrm{d}\beta \int_{\R}  \xi^2 \, [e^{\imath \al \xi X}\otimes  e^{\imath \beta\xi X} \otimes  e^{\imath (1-\al-\beta)\xi X}]\, \mu_f(\mathrm{d}\xi) \quad \in \ca \hat{\otimes} \ca \hat{\otimes} \ca\ .$$
\end{definition}

\subsection{Filtration and Hölder topologies}\label{subsec:gubi} 

From now on and for the rest of Section \ref{sec:general-rp}, we fix a process $X:[0,T]\to \ca_\ast$ and assume that $X$ is $\ga$-Hölder regular, that is
$$\sup_{0\leq s<t\leq T}\frac{\|X_t-X_s\|}{|t-s|^\ga} \ < \ \infty \ ,$$
for some fixed coefficient $\ga\in (1/3,1/2)$.

\smallskip

With this process in hand, we denote by $\{\ca_t\}_{t\in[0,T]}=\{\ca^X_t\}_{t\in[0,T]}$ the \emph{filtration} generated by $X$, that is, for each $t\in [0,T]$, $\ca_t$ stands for the closure (with respect to the operator norm) of the unital subalgebra of $\ca$ generated by $\{X_s\}_{0\leq s\leq t}$.

\smallskip

For any fixed interval $I\subset [0,T]$, a process $Y: I\to \ca$ is said to be \emph{adapted} if for each $t\in I$, $Y_t\in \ca_t$. In the same way, a biprocess $\bu:[0,T]\to \ca \hat{\otimes}\ca$, resp. a triprocess $\mathcal{U}:[0,T]\to \ca \hat{\otimes} \ca \hat{\otimes}\ca$, is \emph{adapted} if for each $t\in [0,T]$, $\bu_t\in \ca_t \hat{\otimes}\ca_t$, resp. $\mathcal{U}_t\in \ca_t \hat{\otimes} \ca_t \hat{\otimes}\ca_t$.

\smallskip

Let us now briefly recall the topologies involved in the rough-path procedure, as far as time-roughness is concerned (and following Gubinelli's approach \cite{gubi}). For $V:=\ca^{\hat{\otimes}n}$ ($n\geq 1$), let $\cac_1(I;V)$ be the set of continuous $V$-valued maps on $I$, and $\cac_2(I;V)$ the set of continuous $V$-valued maps on the simplex $\cs_2:=\{(s,t)\in I^2: \ s\leq t\}$ that vanish on the diagonal. The increments of a path $g\in \cac_1(I;V)$ will be denoted by $\der g_{st}:=g_t-g_s$ ($s\leq t$) and for every $\al\in (0,1)$, we define the $\al$-Hölder spaces $\cac_1^\al(I;V)$, resp. $\cac_2^\al(I;V)$, as
$$\cac_1^\al(I;V):=\big\{ h \in \cac_1(I;V): \ \cn[h;\cac_1^\al(I;V)]:= \sup_{s<t \in I} \frac{\|\der h_{st}\|}{\lln t-s \rrn^\al} \ < \infty \big\} \ ,$$
resp.
$$\cac_2^\al(I;V):=\big\{h \in \cac_2(I;V): \  \cn[h;\cac_2^\al(I;V)]:=\sup_{s<t \in I} \frac{\| h_{st}\|}{\lln t-s \rrn^\al} \ < \infty \big\} \ .$$

\subsection{The product Lévy area}\label{subsec:levy}
Consider the successive spaces
$$\cl_T(\ca_{\rightharpoonup}):=\{L=(L_{st})_{0\leq s < t \leq T}: \ L_{st}\in \cl(\ca_s\hat{\otimes} \ca_s,\ca)\} \ ,$$
$$\cl_T(\ca_{\to}):=\{L=(L_{st})_{0\leq s < t \leq T}: \ L_{st}\in \cl(\ca_s\hat{\otimes} \ca_s,\ca_t)\} \ ,$$
and for every $\la \in [0,1]$, denote by $\cac_2^\la(\cl_T(\ca_{\rightharpoonup}))$, resp. $\cac_2^\la(\cl_T(\ca_\to))$, the set of elements $L\in \cl_T(\ca_{\rightharpoonup})$, resp. $L\in \cl_T(\ca_\to)$, for which the following quantity is finite:
\begin{equation}\label{defi-norm-levy area}
\cn[L;\cac_2^\la(\cl_T(\ca_{\rightharpoonup}))]:=\sup_{\substack{s<t\in [0,T]\\ \bu\in \ca_s \hat{\otimes} \ca_s,\bu\neq 0}}\frac{\|L_{st}[\bu]\|}{\lln t-s\rrn^\la \| \bu\|} \ .
\end{equation}

\smallskip

At this point, recall that we have fixed a $\ga$-Hölder process $X:[0,T]\to \ca_\ast$ ($\ga \in (1/3,1/2)$) for the whole section \ref{sec:general-rp}.

\begin{definition}\label{defi:aire-gene}
We call \emph{product Lévy area} above $X$ any process $\mathbb{X}^{2}$ such that:

\smallskip

\noindent
(i) ($2\ga$-roughness) $\mathbb{X}^2\in \cac_2^{2\ga}(\cl_T(\ca_\to))$,

\smallskip

\noindent
(ii) (Product Chen identity) For all $s<u<t$ and $\mathbf{U}\in \ca_s\hat{\otimes}\ca_s$, 
\begin{equation}\label{chen}
\mathbb{X}^2_{st}[\mathbf{U}]-\mathbb{X}^2_{su}[\mathbf{U}]-\mathbb{X}^2_{ut}[\mathbf{U}]=(\mathbf{U} \sharp \der X_{su}) \, \der X_{ut}\ .
\end{equation}
\end{definition}

\smallskip

\begin{remark}\label{rk:prod-levy-area}
Recall that Definition \ref{defi:aire-gene} is derived from the theoretical analysis performed in \cite[Section 4]{deya-schott} with equation (\ref{equa-diff-q-bm}) in mind. At some heuristic level, and following the classical rough-path approach, the notion of product Lévy area must be seen as some abstract version of the iterated integral 
\begin{equation}\label{lebesgue-levy-area}
\mathbb{X}^2_{st}[\mathbf{U}]= \int_s^t (\mathbf{U} \sharp \der X_{su}) \, \mathrm{d}X_u \ ,
\end{equation}
noting that definition of this integral is not clear a priori for a non-differentiable process $X$. 
As pointed out in \cite{deya-schott}, the above notion of \enquote{Lévy area} is specifically designed to handle the non-commutative algebra dynamics of (\ref{equa-diff-q-bm}), and it offers a much more efficient approach than general rough-path theory based on \enquote{tensor} Lévy areas (the object considered in \cite{lyons-book}). In a commutative setting (i.e., if $\ca$ were a commutative algebra), the basic process $\mathbb{A}_{st}(\bu):=\frac12 (\bu\sharp \delta X_{st}) \, \delta X_{st}$ would immediately provide us with such a product Lévy area. In the general (non-commutative) situation though, this path only satisfies 
$$\mathbb{A}_{st}[\bu]-\mathbb{A}_{su}[\bu]-\mathbb{A}_{ut}[\bu]=\frac12 \big[ (\bu\sharp \delta X_{su}) \, \delta X_{ut}+(\bu\sharp \delta X_{ut}) \, \delta X_{su} \big] \ ,$$
so that $\mathbb{A}$ may not meet the product-Chen condition $(ii)$, making Definition \ref{defi:aire-gene} undoubtedly relevant.
\end{remark}

\subsection{Controlled (bi)processes and integration}\label{subsec:main}
A second ingredient in the rough-path machinery (in addition to a \enquote{Lévy area}) consists in the identification of a suitable class of integrands for the future rough integral with respect to $X$. The following definition naturally arises in this setting:

\begin{definition}\label{def:control-proc}
Given a time interval $I\subset [0,T]$, we call \emph{adapted controlled process}, resp. \emph{biprocess}, on $I$ any adapted process $Y\in \cac_1^\ga(I;\ca)$, resp. biprocess  $\bu \in \cac_1^\ga(I;\ca \hat{\otimes} \ca)$, with increments of the form
\begin{equation}\label{decompo-gene}
(\der Y)_{st}=\by^X_s \sharp (\der X)_{st}+Y^\flat_{st}\ , \quad s<t\in I \ ,
\end{equation}
resp.
\begin{equation}\label{decompo-bipro}
(\der \bu)_{st}=(\der X)_{st} \sharp \mathcal{U}_s^{X,1} +\mathcal{U}_s^{X,2} \sharp (\der X)_{st}+ \bu^\flat_{st} \ , \quad s<t\in I \ ,
\end{equation}
for some adapted biprocess $\by^X \in \cac_1^{\ga}(I;\ca \hat{\otimes} \ca)$ , resp. adapted triprocesses $\mathcal{U}^{X,1},\mathcal{U}^{X,2}\in \cac_1^{\ga}(I;\ca^{\hat{\otimes}3})$, and $Y^\flat \in \cac_2^{2\ga}(I;\ca)$, resp. $\bu^\flat \in \cac_2^{2\ga}(I;\ca \hat{\otimes} \ca)$. We denote by $\cq_X(I)$, resp. $\mathbf{Q}_X(I)$, the space of adapted controlled processes, resp. biprocesses, on $I$, and finally we define $\cq_X^\ast(I)$ as the subspace of controlled processes $Y \in \cq_X(I)$ for which one has both $Y^\ast_s=Y_s$ and $(\by^X_s)^\ast=\by^X_s$ for every $s\in I$.
\end{definition}

\smallskip

\begin{example}\label{ex-1}
If $f,g \in \mathbf{F}_2$ and $Y\in \cq_X^\ast(I)$ with decomposition (\ref{decompo-gene}), then $\bu:=f(Y) \otimes g(Y) \in \mathbf{Q}_X(I)$ with
$$\mathcal{U}^{X,1}_s := [\partial f(Y_s) \, \by^X_s ] \otimes g(Y_s) \quad , \quad \mathcal{U}^{X,2}_s=f(Y_s) \otimes [\partial g(Y_s) \, \by^X_s] \ .$$
\end{example}
\begin{example}\label{ex-2}
If $f\in \mathbf{F}_3$, then $\bu:=\partial f(X) \in \mathbf{Q}_X([0,T])$ with $\mathcal{U}^{X,1}_s =\mathcal{U}^{X,2}_s=\partial^2 f(X_s)$.
\end{example}

\smallskip

We are finally in a position to recall the definition of the \emph{rough integral} with respect to $X$, which can be expressed (among other ways) as the limit of \enquote{corrected Riemann sums}:

\begin{proposition}\label{prop-int-gen}
\cite[Proposition 4.12]{deya-schott} Assume that we are given a product Lévy area $\mathbb{X}^2$ above $X$, in the sense of Definition \ref{defi:aire-gene}, as well as a time interval $I=[\ell_1,\ell_2]\subset [0,T]$. Then for every $\bu\in \mathbf{Q}_X(I)$ with decomposition (\ref{decompo-bipro}), all $s<t\in I$ and every subdivision $D_{st} = \{t_0=s<t_1 <\ldots<t_n=t\}$ of $[s,t]$ with mesh $|D_{st}|$ tending to $0$, the corrected Riemann sum
$$\sum_{t_i\in D_{st}} \Big\{ \bu_{t_i} \sharp (\der X)_{t_it_{i+1}}+[\mathbb{X}^2_{t_it_{i+1}}\times \id](\mathcal{U}^{X,1}_{t_i})+[\id \times \mathbb{X}^{2,\ast}_{t_it_{i+1}}](\mathcal{U}^{X,2}_{t_i})\Big\}$$
converges in $\ca$ as $|D_{st}| \to 0$. We call the limit the \emph{rough integral} (from $s$ to $t$) of $\bu$ against $\mathbb{X}:=(X,\mathbb{X}^2)$, and we denote it by $\int_s^t \bu_u \sharp \mathrm{d}\mathbb{X}_u$. This construction satisfies the two following properties: 

\smallskip

\noindent
$\bullet$ \text{(Consistency)} If $X$ is a differentiable process in $\ca$ and $\mathbb{X}^2$ is understood in the classical Lebesgue sense (that is, as in (\ref{lebesgue-levy-area})), then $\int_s^t \bu_u \sharp \mathrm{d}\mathbb{X}_u$ coincides with the classical Lebesgue integral $\int_s^t [ \bu_u \sharp X'_u] \, du$;

\smallskip

\noindent
$\bullet$ \text{(Stability)} For every $A\in \ca$, there exists a unique process $Z\in \cq_X(I)$ such that $Z_{\ell_1}=A$ and $(\der Z)_{st}=\int_s^t \bu_u \sharp \mathrm{d}\mathbb{X}_u$ for all $s<t\in I$.
\end{proposition}

\begin{theorem}\label{sol}
\cite[Theorem 4.15]{deya-schott} Assume that we are given a product Lévy area $\mathbb{X}^2$ above $X$. Let $f=(f_1,\ldots,f_m) \in\mathbf{F}_3^m$, $g=(f_1^\ast,\ldots,f_m^\ast)$ or $(f_m^\ast,\ldots,f_1^\ast)$, and fix $A \in \ca_\ast$. Then the equation
\begin{equation}\label{sys-base}
Y_{0}=A \quad , \quad (\der Y)_{st}=\sum_{i=1}^m \int_s^t f_i(Y_u) \, \mathrm{d}\mathbb{X}_u \, g_i(Y_u)\ , \quad s<t \in [0,T]\ ,
\end{equation}
interpreted with Proposition \ref{prop-int-gen}, admits a unique solution $Y\in \cq_X^\ast([0,T])$.
\end{theorem}

\subsection{Approximation results}
Another advantage of the rough-path approach - beyond its consistency and stability properties - lies in the continuity of the constructions with respect to the driving (rough) path. In this non-commutative setting, and following the approach of \cite{deya-schott}, the phenomenon can be illustrated through several \enquote{Wong-Zakaï-type} approximation results, which we propose to briefly review here. To this end, for every sequence of partitions $(D^n)$ of $[0,T]$ with mesh tending to zero, denote by $\{X^n_t\}_{t\in [0,T]}=\{X^{D^n}_t\}_{t\in [0,T]}$ the sequence of linear interpolations of $X$ along $D^n$, i.e., if $D^n:=\{0=t_0< t_1<\ldots <t_k=T\}$, 
$$X^n_t:=X_{t_i}+\frac{t-t_i}{t_{i+1}-t_i} \delta X_{t_i t_{i+1}} \quad \text{for} \ t\in [t_i,t_{i+1}]\ .$$
Then consider the sequence of approximated product Lévy areas defined for every $\bu\in \ca \hat{\otimes} \ca$ as
\begin{equation}\label{levy-area-appr}
\mathbb{X}^{2,n}_{st}[\bu]=\mathbb{X}^{2,D^n}_{st}[\bu]:=\int_s^t (\bu\sharp \der X^n_{su}) \, \mathrm{d}X^n_u \ ,\quad  s<t\in [0,T]\ ,
\end{equation}
where the integral is understood in the classical Lebesgue sense. In other words, if $t_k \leq s <t_{k+1}\leq t_\ell \leq t < t_{\ell+1}$,
\begin{align*}
&\mathbb{X}^{2,n}_{st}[\bu]=\int_s^{t_{k+1}} \frac{du}{t_{k+1}-t_k} \big( \bu \sharp \delta X^n_{su} \big) (\delta X)_{t_k t_{k+1}}\\
&\hspace{1cm}+\sum_{i=k+1}^{\ell-1}\int_{t_i}^{t_{i+1}} \frac{du}{t_{i+1}-t_i} \big( \bu \sharp \delta X^n_{su} \big) (\delta X)_{t_i t_{i+1}}+\int_{t_\ell}^{t} \frac{du}{t_{k+1}-t_k} \big( \bu \sharp \delta X^n_{su} \big) (\delta X)_{t_k t_{k+1}} \ .
\end{align*}

\begin{proposition}\label{pro-approx-1}
\cite[Proposition 4.16]{deya-schott} Assume that there exists a product Lévy area $\mathbb{X}^2$ above $X$ such that, as $n$ tends to infinity,
\begin{equation}\label{cond-conv}
\cn[X^n - X ;\cac_1^\ga([0,T];\ca)]\to 0 \quad \text{and} \quad \cn[\mathbb{X}^{2,n} -\mathbb{X}^2; \cac_2^{2\ga}(\cl_T(\ca_{\rightharpoonup}))]\to 0 \ .
\end{equation}
Then for all $f,g \in  \mathbf{F}_3$, it holds that 
\begin{equation}\label{resu-conv-1}
\int_.^. f(X^n_u) \, \mathrm{d}X^n_u \, g(X^n_u) \xrightarrow{n\to \infty} \int_.^. f(X_u) \, \mathrm{d}\mathbb{X}_u \, g(X_u) \quad \text{in} \ \ \cac_2^\ga([0,T];\ca)\ ,
\end{equation}
where the integral in the limit is interpreted with Proposition \ref{prop-int-gen}. Similarly, for all $f \in \mathbf{F}_3$, one has
\begin{equation}\label{resu-conv-2}
\int_.^. \partial f(X^n_u) \sharp \mathrm{d}X^n_u \xrightarrow{n\to \infty} \int_.^. \partial f (X_u) \sharp \mathrm{d}\mathbb{X}_u \quad \text{in} \ \ \cac_2^\ga([0,T];\ca)\ ,
\end{equation}
which immediately yields Itô's formula: for all $s<t\in [0,T]$,
\begin{equation}
\der f(X)_{st}=\int_s^t \partial f (X_u) \sharp \mathrm{d}\mathbb{X}_u \ .
\end{equation}
\end{proposition} 

Finally, for some fixed $f=(f_1,\ldots,f_m) \in \mathbf{F}_3^m$ and $g=(f_1^\ast,\ldots,f_m^\ast)$ (or $g:=(f_m^\ast,\ldots,f_1^\ast)$), let us denote by $Y^n=Y^{D^n}$ the solution of the classical Lebesgue equation on $[0,T]$
$$Y^n_0=A\in \ca_\ast \quad , \quad dY^n_t =\sum\nolimits_{i=1}^m f_i(Y^n_t) \, \mathrm{d} X^n_t \, g_i(Y^n_t)\ .$$

\begin{theorem}\label{theo-approx}
\cite[Theorem 4.17]{deya-schott} Under the assumptions of Proposition \ref{pro-approx-1}, one has $Y^n \xrightarrow{n\to \infty} Y$ in $\cac_1^\ga([0,T];\ca)$, where $Y$ is the solution of (\ref{sys-base}) given by Theorem \ref{sol}.
\end{theorem} 

As we pointed it out in the introduction, these convergence results are based on Taylor-expansion procedures and accordingly, the consideration of an algebra norm for the control of $\bu$ and $L_{st}[\bu]$ in the roughness assumption (\ref{defi-norm-levy area}) is an essential ingredient.

\section{A product Lévy area above the $q$-Brownian motion}\label{sec:levy-area}

We go back here to the $q$-Bm setting described in Section \ref{sec:intro}. Namely, we fix $q\in [0,1)$ and consider a $q$-Brownian motion $(X_t)_{t\geq 0}$ in some non-commutative probability space $(\ca,\vp)$. With the developments of the previous section in mind, the route towards an efficient operator-norm calculus for $X$ is now clear: we need to exhibit a product Lévy area above $X$, in the sense of Definition \ref{defi:aire-gene}. Our main result thus reads as follows: 

\begin{theorem}\label{theo:exi-levy-area-qbm}
Denote by $\{X^n_t\}_{t\geq 0}$ the linear interpolation of $X$ along the dyadic partition $D^n:=\{t_i^n \, , \, i\geq 0\}$, $t_i^n:=\frac{i}{2^n}$. Then there exists a product Lévy area $\mathbb{X}^{2,S}$ above $X$, in the sense of Definition \ref{defi:aire-gene}, such that for every $T>0$ and every $0<\ga<1/2$, one has 
\begin{equation}\label{conv-theo}
X^n \to X \quad \text{in} \ \ \cac_1^\ga([0,T];\ca) \quad \text{and} \quad \mathbb{X}^{2,n} \to \mathbb{X}^{2,S} \quad \text{in} \ \ \cac_2^{2\ga}(\cl_T(\ca_\to)) \ ,
\end{equation}
where $\mathbb{X}^{2,n}$ is defined by (\ref{levy-area-appr}). We call $\mathbb{X}^{2,S}$ the Stratonovich product Lévy area above $X$. 
\end{theorem}

Based on this result, the conclusions of Proposition \ref{prop-int-gen}, Theorem \ref{sol}, Proposition \ref{pro-approx-1} and Theorem \ref{theo-approx} can all be applied to the $q$-Brownian motion, with limits understood as rough integrals with respect to the \enquote{product rough path} $\mathbb{X}^S:=(X,\mathbb{X}^{2,S})$. The Stratonovich terminology is here used as a reference to the classical commutative situation, where the (almost sure) limit of the sequence of approximated Lévy areas would indeed coincide with the Stratonovich iterated integral (see also Corollary \ref{coro:identif-strato} for another justification of this terminology).

\smallskip

Before we turn to the proof of Theorem \ref{theo:exi-levy-area-qbm}, let us recall that the whole difficulty in constructing a stochastic integral with respect to the general $q$-Bm, in comparison with the free ($q=0$) or the commutative ($q\to 1$) cases, lies in the absence of any satisfying \enquote{$q$-freeness} property for the increments of the process when $q\in (0,1)$ (as reported by Speicher in \cite{speicher}). For instance, if $s<u<t$,
$$
\vp\big((X_u-X_s)(X_t-X_u)(X_u-X_s)(X_t-X_u)\big) =q\,  \vp\big((X_u-X_s)^2\big)\vp\big((X_t-X_u)^2\big)=q\, |u-s| |t-u|  \ ,
$$
which shows that, for $q\neq 0$, the disjoint increments of a $q$-Brownian motion $\{X_t\}_{t\geq 0}$ are indeed not freely independent (in the sense of \cite[Definition 2.6]{deya-schott}), making most of the arguments of \cite{biane-speicher} unexploitable in this situation.

\smallskip

This being said, we can still rely here on the basic fact that for all $q\in [0,1)$, $\vp\big((X_u-X_s)(X_t-X_u)\big)=0$. Together with Formula (\ref{form-q-gaussian}), this very weak \enquote{freeness} property of the increments will somehow be sufficient for our purpose, the construction of a product Lévy area being much more specific than the construction of a general stochastic integral (along Itô's standard procedure).

\

The proof of Theorem \ref{theo:exi-levy-area-qbm} will also appeal to the two following elementary lemmas. The first one (whose proof follows immediately from (\ref{form-q-gaussian})) is related to the linear stability of $q$-Gaussian families:

\begin{lemma}\label{lem:stabi-q-gauss}
For any fixed $q\in [0,1)$, let $Y:=\{Y_1,\ldots,Y_d\}$ be a $q$-Gaussian vector in some non-commutative probability space $(\ca,\vp)$, and consider a real-valued $(d\times m)$-matrix $\Lambda$. Then $Z:=\Lambda Y$ is also a $q$-Gaussian vector in $(\ca,\vp)$.
\end{lemma}

We will also need the following general topology property on the space accommodating any Lévy area:

\begin{lemma}\label{lem:complete-space}
The space $\cac_2^\la(\cl_T(\ca_{\rightharpoonup}))$, endowed with the norm (\ref{defi-norm-levy area}), is complete.
\end{lemma}

\begin{proof}
Although the arguments are classical, let us provide a few details here, since the $\cac_2^\la(\cl_T(\ca_{\rightharpoonup}))$-structure is not exactly standard. 

\smallskip

Consider a Cauchy sequence $L^n$ in $\cac_2^\la(\cl_T(\ca_{\rightharpoonup}))$. For every fixed $s\in [0,T]$, the sequence $L^n_{s.}$ defines a Cauchy sequence in the space $L^\infty([s,T]; \cl(\ca_s \hat{\otimes}\ca_s,\ca))$ of bounded functions on $[s,T]$ (with values in $\cl(\ca_s \hat{\otimes}\ca_s,\ca)$), endowed with the uniform norm. Therefore it converges in the latter space to some function $L_{s.}$. The fact that the so-defined family $\{L_{st}\}_{s<t}$ belongs to $\cac_2^\la(\cl_T(\ca_{\rightharpoonup}))$ is an immediate consequence of the boundedness of $L^n$ in $\cac_2^\la(\cl_T(\ca_{\rightharpoonup}))$. Finally, given $\varepsilon >0$ and for all fixed $s<t$, we know that there exists $M_{\varepsilon,s,t}\geq 0$ such that for all $m\geq M_{\varepsilon,s,t}$, $\|L^m_{st}-L_{st}\|_{\cl(\ca_s \hat{\otimes}\ca_s,\ca)}\leq \frac{\varepsilon}{2} |t-s|^\la$. On the other hand, there exists $N_\varepsilon \geq 0$ such that for all $n,m\geq N_\varepsilon$ and all $s<t$, $\|L^n_{st}-L^m_{st}\|_{\cl(\ca_s \hat{\otimes}\ca_s,\ca)}\leq \frac{\varepsilon}{2} |t-s|^\la$. Therefore, for all $n\geq N_\varepsilon$ and all $s<t$, we get that for $m:=\max(N_\varepsilon,M_{\varepsilon,s,t})$,
$$\|L^n_{st}-L_{st}\|_{\cl(\ca_s \hat{\otimes}\ca_s,\ca)}\leq \|L^n_{st}-L^m_{st}\|_{\cl(\ca_s \hat{\otimes}\ca_s,\ca)}+\|L_{st}-L^m_{st}\|_{\cl(\ca_s \hat{\otimes}\ca_s,\ca)}\leq  \varepsilon |t-s|^\la\ ,$$
and so $L^n\to L$ in $\cac_2^\la(\cl_T(\ca_{\rightharpoonup}))$, which achieves to prove that the latter space is complete.
\end{proof}

\

\begin{proof}[Proof of Theorem \ref{theo:exi-levy-area-qbm}]
Throughout the proof, we will denote by $A\lesssim B$ any bound of the form $A\leq c B$, where $c$ is a universal constant independent from the parameters under consideration. The first-order convergence statement in (\ref{conv-theo}) is a straightforward consequence of the $1/2$-Hölder regularity of $X$. In fact, using (\ref{q-bm-holder}), it can be checked that for all $n\geq 0$ and $s<t$, 
\begin{equation}\label{fir-or-pr}
\| \delta X^n_{st}\|\lesssim \|X_1\|  |t-s|^{1/2}  \quad \text{and} \quad \| \delta(X^n-X)_{st}\|\lesssim \|X_1\|  |t-s|^{\gamma} 2^{-n(1/2-\gamma)}\ .
\end{equation}
Let us turn to the second-order convergence and to this end, fix $n\geq 0$ and $s<t$ such that $t_k^n \leq s<t_{k+1}^n$, $t_\ell^n \leq t < t_{\ell+1}$, with $k\leq \ell$. If $|\ell-k|\leq 1$, or in other words if $|t-s|\leq 2^{-n+1}$, the expected bound can be readily derived from the first estimate in (\ref{fir-or-pr}), that is for every $\bu\in \ca_s \hat{\otimes}\ca_s$, we get from (\ref{fir-or-pr})
$$\|\mathbb{X}^{2,n+1}_{st}[\bu]-\mathbb{X}^{2,n}_{st}[\bu]\|\leq \|\mathbb{X}^{2,n+1}_{st}[\bu]\|+\|\mathbb{X}^{2,n}_{st}[\bu]\| \lesssim \|X_1\|^2  |t-s|^{2\gamma} 2^{-n(1/2-\gamma)} \ . $$
Assume from now on that $\ell \geq k+2$ and in this case consider the decomposition, for every $\bu\in \ca_s \hat{\otimes}\ca_s$,
\begin{eqnarray}
\lefteqn{\mathbb{X}^{2,n+1}_{st}[\bu]-\mathbb{X}^{2,n}_{st}[\bu]}\nonumber\\
&=& \bigg[ \int_{t_{k+1}^n}^{t_\ell^n} \big( \bu \sharp \delta X^{n+1}_{t_{k+1}^n u}\big) \, \mathrm{d}X^{n+1}_u-\int_{t_{k+1}^n}^{t_\ell^n} \big( \bu \sharp \delta X^{n}_{t_{k+1}^nu}\big) \, \mathrm{d}X^{n}_u \bigg]\nonumber\\
& & + \bigg[ \int_s^{t_{k+1}^n}  \big( \bu \sharp \delta X^{n+1}_{su}\big) \, \mathrm{d}X^{n+1}_u +\int_{t_\ell^n}^{t} \big( \bu \sharp \delta X^{n+1}_{su}\big) \, \mathrm{d}X^{n+1}_u -\int_s^{t_{k+1}^n} \big( \bu \sharp \delta X^{n}_{su}\big) \, \mathrm{d}X^{n}_u\nonumber \\
& & \hspace{9cm}-\int_{t_\ell^n}^{t} \big( \bu \sharp \delta X^{n}_{su}\big) \, \mathrm{d}X^{n}_u\bigg]\nonumber\\
& & +\bigg[\int_{t_{k+1}^n}^{t_\ell^n} \big( \bu \sharp \delta X^{n+1}_{st_{k+1}^n}\big) \, \mathrm{d}X^{n+1}_u-\int_{t_{k+1}^n}^{t_\ell^n} \big( \bu \sharp \delta X^{n}_{st_{k+1}^n}\big) \, \mathrm{d}X^{n}_u    \bigg] \ .\label{starting-decompo-order-two}
\end{eqnarray}
The \enquote{boundary} integrals within the second and third brackets can again be bounded individually using the first estimate in (\ref{fir-or-pr}) only. For instance, 
\begin{eqnarray*}
\lefteqn{\Big\| \int_{t_\ell^n}^{t} \big( \bu \sharp \delta X^{n+1}_{su}\big) \, \mathrm{d}X^{n+1}_u\Big\|}\\
&\lesssim &  \|X_1\| \|\bu\|\bigg[   \1_{\{t_{2\ell}^{n+1} \leq t< t_{2\ell+1}^{n+1}\}} \int_{t_\ell^n}^t |s-u|^{1/2}  \big( 2^{n+1} \|\delta X_{t_{2\ell}^{n+1}t_{2\ell+1}^{n+1}}\|\big) \\
& &\hspace{1.5cm} +\1_{\{t_{2\ell+1}^{n+1} \leq t< t_{2\ell+2}^{n+1}\}}\int_{t_{2\ell}^{n+1}}^{t_{2\ell+1}^{n+1}} |s-u|^{1/2}   \big( 2^{n+1} \|\delta X_{t_{2\ell}^{n+1}t_{2\ell+1}^{n+1}}\|\big)\\
& &\hspace{1.5cm} +\1_{\{t_{2\ell+1}^{n+1} \leq t< t_{2\ell+2}^{n+1}\}}\int_{t_{2\ell+1}^{n+1}}^t |s-u|^{1/2}  \big( 2^{n+1} \|\delta X_{t_{2\ell+1}^{n+1}t_{2\ell+2}^{n+1}}\|\big)\bigg]\\
&\lesssim & \|X_1\|^2 \|\bu\| |t-s|^{2\ga} 2^{-n(1/2-\ga)} \ .
\end{eqnarray*}

\smallskip

Therefore, we only have to focus on the first bracket in decomposition (\ref{starting-decompo-order-two}). In fact, noting that
$$\int_{t_i^n}^{t_{i+1}^n}\big( \bu \sharp \delta X^{n}_{t_{i}^nu}\big) \, \mathrm{d}X^{n}_u =\frac12 \big( \bu \sharp \delta X_{t_i^n t_{i+1}^n} \big) \, \delta X_{t_i^n t_{i+1}^n} \ ,$$
we get 
\begin{eqnarray}
\lefteqn{\int_{t_{k+1}^n}^{t_\ell^n} \big( \bu \sharp \delta X^{n+1}_{t_{k+1}^n u}\big) \, \mathrm{d}X^{n+1}_u-\int_{t_{k+1}^n}^{t_\ell^n} \big( \bu \sharp \delta X^{n}_{t_{k+1}^nu}\big) \, \mathrm{d}X^{n}_u} \nonumber\\
&=& \sum_{i=k+1}^{\ell-1} \bigg\{\int_{t_{2i}^{n+1}}^{t_{2i+1}^{n+1}} \big( \bu \sharp \delta X^{n+1}_{t_{k+1}^n u}\big) \, \mathrm{d}X^{n+1}_u+\int_{t_{2i+1}^{n+1}}^{t_{2i+2}^{n+1}} \big( \bu \sharp \delta X^{n+1}_{t_{k+1}^n u}\big) \, \mathrm{d}X^{n+1}_u \nonumber\\
& &\hspace{8cm}-\int_{t_{i}^n}^{t_{i+1}^n} \big( \bu \sharp \delta X^{n}_{t_{k+1}^nu}\big) \, \mathrm{d}X^{n}_u \bigg\}\nonumber\\
&=& \sum_{i=k+1}^{\ell-1} \bigg\{\bigg[\int_{t_{2i}^{n+1}}^{t_{2i+1}^{n+1}} \big( \bu \sharp \delta X^{n+1}_{t_{2i}^{n+1} u}\big) \, \mathrm{d}X^{n+1}_u+\int_{t_{2i+1}^{n+1}}^{t_{2i+2}^{n+1}} \big( \bu \sharp \delta X^{n+1}_{t_{2i+1}^{n+1} u}\big) \, \mathrm{d}X^{n+1}_u \nonumber\\
& &\hspace{8cm}-\int_{t_{i}^n}^{t_{i+1}^n} \big( \bu \sharp \delta X^{n}_{t_{i}^nu}\big) \, \mathrm{d}X^{n}_u\bigg]\nonumber\\
& &\hspace{2cm} +\Big[ \big( \bu \sharp \delta X_{t_{k+1}^n t_{2i}^{n+1}} \big) \, \delta X_{t_{2i}^{n+1} t_{2i+1}^{n+1}}+\big( \bu \sharp \delta X_{t_{k+1}^n t_{2i+1}^{n+1}} \big) \, \delta X_{t_{2i+1}^{n+1} t_{2i+2}^{n+1}}\nonumber \\
& &\hspace{8cm}-\big( \bu \sharp \delta X_{t_{k+1}^n t_{2i}^{n+1}} \big) \, \delta X_{t_{2i}^{n+1} t_{2i+2}^{n+1}} \Big] \bigg\}\nonumber\\
&=& \sum_{i=k+1}^{\ell-1} \bigg\{\frac12\bigg[\big( \bu \sharp \delta X_{t_{2i}^{n+1} t_{2i+1}^{n+1}} \big) \, \delta X_{t_{2i}^{n+1} t_{2i+1}^{n+1}}+\big( \bu \sharp \delta X_{t_{2i+1}^{n+1}t_{2i+2}^{n+1}} \big)\delta X_{t_{2i+1}^{n+1}t_{2i+2}^{n+1}} \nonumber\\
& &\hspace{7cm}-\big( \bu \sharp \delta X_{t_{2i}^{n+1}t_{2i+2}^{n+1}} \big) \delta X_{t_{2i}^{n+1}t_{2i+2}^{n+1}} \bigg]\nonumber\\
& &\hspace{2.5cm} +\Big[ -\big( \bu \sharp \delta X_{t_{k+1}^n t_{2i}^{n+1}} \big) \, \delta X_{t_{2i+1}^{n+1} t_{2i+2}^{n+1}}+\big( \bu \sharp \delta X_{t_{k+1}^n t_{2i+1}^{n+1}} \big) \, \delta X_{t_{2i+1}^{n+1} t_{2i+2}^{n+1}}\Big]\bigg\}\nonumber \\
&=& \sum_{i=k+1}^{\ell-1} \bigg\{\frac12\bigg[-\big( \bu \sharp \delta X_{t_{2i}^{n+1} t_{2i+1}^{n+1}} \big) \, \delta X_{t_{2i+1}^{n+1} t_{2i+2}^{n+1}}-\big( \bu \sharp \delta X_{t_{2i+1}^{n+1}t_{2i+2}^{n+1}} \big)\delta X_{t_{2i}^{n+1}t_{2i+1}^{n+1}} \bigg]\nonumber\\
& &\hspace{7cm} +\Big[ \big( \bu \sharp \delta X_{t_{2i}^{n+1}t_{2i+1}^{n+1}} \big) \, \delta X_{t_{2i+1}^{n+1} t_{2i+2}^{n+1}}\Big]\bigg\}\nonumber \\
&=&\frac12 \sum_{i=k+1}^{\ell-1} \bigg[ \Big( \bu \sharp (\delta X)_{t_{2i}^{n+1}t_{2i+1}^{n+1}} \Big) \, (\delta X)_{t_{2i+1}^{n+1}t_{2i+2}^{n+1}}-\Big( \bu \sharp (\delta X)_{t_{2i+1}^{n+1}t_{2i+2}^{n+1}} \Big) \, (\delta X)_{t_{2i}^{n+1}t_{2i+1}^{n+1}} \bigg]  .\label{main-term-area}
\end{eqnarray}
Let us bound the two sums
$$S^{1,n}_{st}[\bu]:= \sum_{i=k+1}^{\ell-1} \Big( \bu \sharp (\delta X)_{t_{2i}^{n+1}t_{2i+1}^{n+1}} \Big) \, (\delta X)_{t_{2i+1}^{n+1}t_{2i+2}^{n+1}}$$
and
$$S^{2,n}_{st}[\bu]:=\sum_{i=k+1}^{\ell-1} \Big( \bu \sharp (\delta X)_{t_{2i+1}^{n+1}t_{2i+2}^{n+1}} \Big) \, (\delta X)_{t_{2i}^{n+1}t_{2i+1}^{n+1}}$$
separately.

\smallskip

\noindent
Consider first the case where $\bu=\sum_{j=1}^o U_j \otimes V_j$, with 
$$U_j:=X_{s^j_1} \cdots X_{s^j_{m_j}} \quad , \quad V_j:=X_{s^j_{m_j+1}} \cdots X_{s^j_{m_j+p_j}} \ ,$$
and $s^j_p \leq s$ for all $j,p$. Besides, let us set $Y_i=Y_{i,n}:=(\delta X)_{t_{i}^{n+1}t_{i+1}^{n+1}}$. With these notations, and for every $r\geq 1$, we have
\begin{eqnarray}
\lefteqn{\vp\big( |S^{1,n}_{st}[\bu]|^{2r} \big) \ =\ \vp\bigg( \bigg[ \bigg( \sum_{i_1}\sum_{j_1} U_{j_1}Y_{2i_1}V_{j_1}Y_{2i_1+1}\bigg) \bigg( \sum_{i_2} \sum_{j_2} U_{j_2}Y_{2i_2}V_{j_2}Y_{2i_2+1}\bigg)^\ast \bigg]^r\bigg)}\nonumber\\
&=& \sum_{i_1,\ldots,i_{2r}}\sum_{j_1,\ldots,j_{2r}} \vp\Big( \big[ U_{j_1}Y_{2i_1}V_{j_1}Y_{2i_1+1}Y_{2i_2+1} V_{j_2}^\ast Y_{2i_2}U_{j_2}^\ast\big]\cdots\nonumber\\
& & \hspace{3cm}\big[ U_{j_{2r-1}}Y_{2i_{2r-1}}V_{j_{2r-1}}Y_{2i_{2r-1}+1}Y_{2i_{2r}+1} V_{j_{2r}}^\ast Y_{2i_{2r}}U_{j_{2r}}^\ast\big] \Big) \ ,\label{joint-moments-proof}
\end{eqnarray}
where each index $i$ runs over $\{k+1,\ldots,\ell-1\}$ and each index $j$ runs overs $\{1,\ldots,o\}$.
At this point, observe that for all fixed $\mathbf{i}:=(i_1,\ldots,i_{2r})$ and $\mathbf{j}:=(j_1,\ldots,j_{2r})$, the family 
$$\{X_{s^{j}_1},\dots,X_{s^{j}_{m_j+p_j}},Y_{2i},Y_{2i+1}, \ i\in \{i_1,\ldots,i_{2r}\}, \ j\in \{j_1,\ldots,j_{2r}\} \}$$
is a $q$-Gaussian family (due to Lemma \ref{lem:stabi-q-gauss}) and accordingly the associated joint moments obey Formula (\ref{form-q-gaussian}). Besides, we have trivially
$$\vp\big( Y_{2i_a}Y_{2i_b+1}\big)=0 \quad , \quad \vp\big( Y_{2i_a}Y_{2i_b}\big)=\vp\big( Y_{2i_a+1}Y_{2i_b+1}\big)=\1_{\{i_a=i_b\}}2^{-(n+1)} \vp\big( |X_1|^2)$$
and
$$\vp\big( Y_{2i} X_{s_a^j})=\vp\big( Y_{2i+1} X_{s_a^j})=0 \ .$$
Using these basic observations and going back to (\ref{joint-moments-proof}), it is clear that, when applying Formula (\ref{form-q-gaussian}) to the expectation in (\ref{joint-moments-proof}), we can restrict the sum to the set of pairings $\pi\in \mathcal{P}_2(\{1,\ldots,N_r\})$ ($N_r:=2\big[(m_{j_1}+p_{j_1})+\ldots +(m_{j_{2r}}+p_{j_{2r}})]+8r$) that decompose - in a unique way - as a combination of three sub-pairings, namely: 1) a pairing $\pi^1\in \mathcal{P}_2(\{1,\ldots,2r\})$ that connects the random variables $\{Y_{2i}\}$ to each other; 2) a pairing $\pi^2\in \mathcal{P}_2(\{1,\ldots,2r\})$ that connects the random variables $\{Y_{2i+1}\}$ to each other; 3) a pairing $\pi^3\in \mathcal{P}_2(\{1,\ldots,N'_r\})$ ($N'_r:=2\big[(m_{j_1}+p_{j_1})+\ldots +(m_{j_{2r}}+p_{j_{2r}})]$) that connects the random variables $\{X_{s_i^j}\}$ to each other. Moreover, with this decomposition in mind, one has clearly
$$\text{Cr}(\pi)\geq \text{Cr}(\pi^1)+\text{Cr}(\pi^2)+\text{Cr}(\pi^3) \ .$$
Consequently, it holds that for all fixed $\mathbf{i}:=(i_1,\ldots,i_{2r})$ and $\mathbf{j}:=(j_1,\ldots,j_{2r})$,
\begin{eqnarray}
\lefteqn{\Big|\vp\Big( \big[ U_{j_1}Y_{2i_1}V_{j_1}Y_{2i_1+1}Y_{2i_2+1} V_{j_2}^\ast Y_{2i_2}U_{j_2}^\ast\big]\cdots}\nonumber\\
& &\hspace{4cm}\big[ U_{j_{2r-1}}Y_{2i_{2r-1}}V_{j_{2r-1}}Y_{2i_{2r-1}+1}Y_{2i_{2r}+1} V_{j_{2r}}^\ast Y_{2i_{2r}}U_{j_{2r}}^\ast\big] \Big)\Big|\nonumber\\
&\leq &\sum_{\substack{\pi^1,\pi^2\in \mathcal{P}_2(\{1,\ldots,2r\}) \\ \pi^3\in \mathcal{P}_2(\{1,\ldots,N'_r\})}} q^{\text{Cr}(\pi^1)+\text{Cr}(\pi^2)+\text{Cr}(\pi^3)} \nonumber\\
& &\hspace{1cm}\prod_{\{a,b\}\in \pi^1} \vp\big(Y_{2i_a}Y_{2i_b}\big) \1_{\{i_a=i_b\}}\prod_{\{c,d\}\in \pi^2} \vp\big( Y_{2i_c+1}Y_{2i_d+1}\big)\1_{\{i_c=i_d\}} \prod_{\{e,f\}\in \pi^3} \vp\big( Z^{\mathbf{j}}_e Z^{\mathbf{j}}_f \big)\label{positivity-assumption}\\
&\leq& 2^{-2r(n+1)} \vp\big( |X_1|^2)^{2r}\bigg( \sum_{\pi^1\in\mathcal{P}_2(\{1,\ldots,2r\})}q^{\text{Cr}(\pi^1)}\prod_{\{a,b\}\in \pi^1} \1_{\{i_a=i_b\}}\bigg)\nonumber\\
& &\bigg( \sum_{\pi^2\in\mathcal{P}_2(\{1,\ldots,2r\})}q^{\text{Cr}(\pi^2)}\prod_{\{c,d\}\in \pi^2}\1_{\{i_c=i_d\}}\bigg)\bigg(\sum_{\pi^3\in \mathcal{P}_2(\{1,\ldots,N'_r\})}q^{\text{Cr}(\pi^3)}\prod_{\{e,f\}\in \pi^3} \vp\big( Z^{\mathbf{j}}_e Z^{\mathbf{j}}_f \big)\bigg)  \ ,\nonumber
\end{eqnarray}
where $Z^{\mathbf{j}}$ stands for the natural reordering of the variables $\{X_{s_m^j}\}$, namely for all $a\in \{1,\ldots,2r\}$ and $b\in \{1,\ldots,m_{j_{a}}+p_{j_{a}}\}$,
$$Z^{\mathbf{j}}_{2[(m_{j_1}+p_{j_1})+\ldots+(m_{j_{a-1}}+p_{j_{a-1}})]+b}=Z^{\mathbf{j}}_{2[(m_{j_1}+p_{j_1})+\ldots+(m_{j_{a-1}}+p_{j_{a-1}})]+[2(m_{j_{a}}+p_{j_{a}})-b]}:=X_{s_b^{j_a}} \ .$$
As a result, the double sum in (\ref{joint-moments-proof}) is bounded by 
\begin{eqnarray}
\lefteqn{2^{-2rn} \vp\big( |X_1|^2)^{2r}\bigg( \sum_{\pi^1\in\mathcal{P}_2(\{1,\ldots,2r\})}q^{\text{Cr}(\pi^1)}\sum_{i_1,\ldots,i_{2r}=k+1}^{\ell-1}\prod_{\{a,b\}\in \pi^1} \1_{\{i_a=i_b\}}\bigg)}\nonumber\\
& &\bigg( \sum_{\pi^2\in\mathcal{P}_2(\{1,\ldots,2r\})}q^{\text{Cr}(\pi^2)}\bigg)\sum_{j_1,\ldots,j_{2r}=1}^o\bigg(\sum_{\pi^3\in \mathcal{P}_2(\{1,\ldots,N'_r\})}q^{\text{Cr}(\pi^3)}\prod_{\{e,f\}\in \pi^3} \vp\big( Z^{\mathbf{j}}_e Z^{\mathbf{j}}_f \big)\bigg) \ .\label{bou-1}
\end{eqnarray}
Now observe that the last sum in (\ref{bou-1}) actually corresponds to
\begin{equation}
\sum_{j_1,\ldots,j_{2r}=1}^o\bigg(\sum_{\pi^3\in \mathcal{P}_2(\{1,\ldots,N'_r\})}q^{\text{Cr}(\pi^3)}\prod_{\{e,f\}\in \pi^3} \vp\big( Z^{\mathbf{j}}_e Z^{\mathbf{j}}_f \big)\bigg)=\vp\Big( \Big| \sum_{j=1}^o U_jV_j \Big|^{2r}\Big) \ ,\label{bou-2}
\end{equation}
and for every fixed $\pi^1\in \mathcal{P}_2(\{1,\ldots,2r\})$,
\begin{eqnarray}
\lefteqn{\sum_{i_1,\ldots,i_{2r}=k+1}^{\ell-1}\prod_{\{a,b\}\in \pi^1} \1_{\{i_a=i_b\}}}\nonumber\\
&=&\bigg(\sum_{i_1,\ldots,i_{2r}=k+1}^{\ell-1}\prod_{\{a,b\}\in \pi^1} \1_{\{i_a=i_b\}}\bigg)^{2(1-2\ga)} \bigg(\sum_{i_1,\ldots,i_{2r}=k+1}^{\ell-1}\prod_{\{a,b\}\in \pi^1} \1_{\{i_a=i_b\}}\bigg)^{4\ga-1}\nonumber \\
&\leq&  (\ell-(k+1))^{2(1-2\ga)r}(\ell-(k+1))^{2r(4\ga-1)}\nonumber\\
& \leq &|t_\ell^n-t_{k+1}^n|^{4r\ga} 2^{4r\ga n} \, \leq \, |t-s|^{4r\ga} 2^{4r\ga n} \, . \label{bou-3}
\end{eqnarray}
By injecting (\ref{bou-2}) and (\ref{bou-3}) into (\ref{bou-1}), we end up with the estimate
\begin{eqnarray*}
\lefteqn{\vp\big( |S^{1,n}_{st}[\bu]|^{2r} \big)}\\
 &\leq& |t-s|^{4r\ga}2^{-2r(1-2\ga)n} \vp\big( |X_1|^2)^{2r}\bigg( \sum_{\pi\in\mathcal{P}_2(\{1,\ldots,2r\})}q^{\text{Cr}(\pi)}\bigg)^2\vp\Big( \Big| \sum_{j=1}^o U_jV_j \Big|^{2r}\Big)\\
&\leq& |t-s|^{4r\ga}2^{-2r(1-2\ga)n} \vp\big( |X_1|^2)^{2r}\vp\big( |X_1|^{2r}\big)^{2}\vp\Big( \Big| \sum_{j=1}^o U_jV_j \Big|^{2r}\Big) \ ,
\end{eqnarray*}
and so
\begin{eqnarray}
\lefteqn{\vp\big( |S^{1,n}_{st}[\bu]|^{2r} \big)^{1/2r}}\nonumber\\
 &\leq&  |t-s|^{2\ga}2^{-(1-2\ga)n} \vp\big( |X_1|^2)\vp\big( |X_1|^{2r}\big)^{1/r}\vp\Big( \Big| \sum_{j=1}^o U_jV_j \Big|^{2r}\Big)^{1/2r} \nonumber\\
&\leq&  |t-s|^{2\ga}2^{-(1-2\ga)n} \|X_1\|^4 \Big\| \sum_{j=1}^o U_jV_j \Big\| \nonumber \\
& \leq& |t-s|^{2\ga}2^{-(1-2\ga)n} \|X_1\|^4 \Big(\sum_{j=1}^o \|U_j\| \|V_j\| \Big) \ .\label{bound-norm-l-2r}
\end{eqnarray}
It is easy to see that the above arguments could also be applied to the more general situation where $\bu:=\sum_{j=1}^o U_j \otimes V_j$ with
$$U_j:=\sum_{k=0}^{K_j} \alpha_{j,k} X_{s^{j,k}_1} \cdots X_{s^{j,k}_{m_{j,k}}}$$
and
$$V_j:=\sum_{\ell=0}^{L_j}\beta_{j,\ell}X_{u^{j,\ell}_{1}} \cdots X_{u^{j,\ell}_{p_{j,\ell}}} \ , \quad \alpha_{j,k},\beta_{j,\ell}\in \C \ , \ s_a^{j,k},u_b^{j,\ell}\in [0,s] \ ,$$
leading in the end to the same bound (\ref{bound-norm-l-2r}). Therefore, this bound (\ref{bound-norm-l-2r}) can actually be extended to any $U_j,V_j \in \ca_s$, which then entails that for every $\bu\in \ca_s \hat{\otimes}\ca_s$,
$$\vp\big( |S^{1,n}_{st}[\bu]|^{2r} \big)^{1/2r}  \leq |t-s|^{2\ga}2^{-(1-2\ga)n} \|X_1\|^4 \|\bu\| \ ,$$
and by letting $r$ tend to infinity, we get by (\ref{prop:norm-op}) that
$$\|S^{1,n}_{st}[\bu]\|\leq |t-s|^{2\ga}2^{-(1-2\ga)n} \|X_1\|^4 \|\bu\| \ .$$
The very same reasoning can of course be used in order to estimate $\|S^{2,n}_{st}[\bu]\|$, with the same resulting bound. Going back to (\ref{starting-decompo-order-two}) and (\ref{main-term-area}), we have thus proved that $\mathbb{X}^{2,n}$ is a Cauchy sequence in $\cac_2^\la(\cl_T(\ca_{\rightharpoonup}))$, and by Lemma \ref{lem:complete-space}, we can therefore assert that it converges in this space to some element $\mathbb{X}^{2,S}$. 

\smallskip

\noindent
The product Chen identity (\ref{lebesgue-levy-area}) for $\mathbb{X}^{2,S}$ is readily obtained by passing to the limit (in a pointwise way) in the product Chen identity that is trivially satisfied by $\mathbb{X}^{2,n}$. Finally, in order to show that $\mathbb{X}^{2,S}$ actually belongs to $\cac_2^\la(\cl_T(\ca_{\to}))$, fix $s<t$, $\bu\in \ca_s \hat{\otimes}\ca_s$, and set
$$W^n:=\mathbb{X}^{2,n}_{st}[\bu] \ , \ W:=\mathbb{X}^{2,S}_{st}[\bu] \ , \ \bar{W}^n:=\int_s^{t_\ell^n} (\bu \sharp \delta X^n_{su})\, \mathrm{d}X^n_u \ ,$$
where $t_\ell^n$ is such that $s<t_\ell^n\leq t <t_{\ell+1}^n$ (considering $n$ large enough). Using the first estimate in (\ref{fir-or-pr}), it is easy to check that $\|W^n-\bar{W}^n \| \to 0$, and so, since $\|W^n-W\|\to 0$, we get that $\|\bar{W}^n - W\|\to 0$. As $\bar{W}^n \in \ca_t$, we can conclude that $W\in \ca_t$, as expected.

\end{proof}

\begin{remark}
Observe that in a commutative setting, the sum (\ref{main-term-area}) would simply vanish, leading to an almost trivial proof, which clearly points out the specificity of our non-commutative framework (as evoked in Remark \ref{rk:prod-levy-area}). 
\end{remark}

\section{Comparison with $L^2(\vp)$-constructions}\label{sec:comparison}

Our objective in this section is to compare the previous $L^\infty(\vp)$-constructions (i.e., constructions based on the operator norm) with the $L^2(\vp)$-constructions exhibited by Donati-Martin in \cite{donati}. In brief, we shall see that, when studied in $L^2(\vp)$, the previous rough constructions correspond to Stratonovich-type integrals, while the constructions in \cite{donati} are more of a Itô-type. This comparison relies on an additional ingredient, the so-called second-quantization operator, whose central role in $q$-integration theory was already pointed out in Donati-Martin's work. 

\smallskip

Since we intend to make specific references to some of the results of \cite{donati}, we assume for simplicity that we are exactly in the same setting as in the latter study. Namely, for a fixed $q\in [0,1)$, we assume that the $q$-Bm $\{X_t\}_{t\geq 0}$ we will handle in this section is constructed as the \enquote{canonical process} on the $q$-Fock space $(\ca,\vp)$ (see \cite{donati} for details on these structures). 

\smallskip 

As in the previous sections, we denote by $\ca_t$ the closure, with respect to the operator norm, of the algebra generated by $\{X_s\}_{s\leq t}$.

\subsection{Second quantization}
Recall that the space $L^2(\vp)$ is defined as the completion of $\ca$ as a Hilbert space through the product
\begin{equation}\label{l-2-product}
\langle U,V\rangle:=\vp( UV^\ast) \ .
\end{equation}
We will denote by $\|.\|_{L^2(\vp)}$ the associated norm, to be distinguished from the operator norm $\|.\|$. For every $t\geq 0$, let $\cb_t$ be the von Neumann algebra generated by $\{X_s\}_{s\leq t}$ (observe in particular that $\ca_t\subset \cb_t \subset \ca$) and denote by $\vp( \cdot |\cb_t)$ the conditional expectation with respect to $\cb_t$. In other words, for every $U\in \ca$, $\vp(U|\cb_t)$ stands for the orthogonal projection of $U$ onto $\cb_t$, with respect to the product (\ref{l-2-product}): $Z=\vp(U|\cb_t)$ if and only if $Z\in \cb_t$ and $\vp(ZW^\ast)=\vp(UW^\ast)$ for every $W\in \cb_t$. 

\smallskip

A possible way to introduce the second-quantization operator goes through the following invariance result:

\begin{lemma}\label{lem:second-quanti}\cite[Theorem 3.1]{donati}.
For all $s_0<t_0$, $s_1<t_1$, with $s_0\leq s_1$, and $U\in \ca_{s_0}\subset \ca_{s_1}$, it holds that
$$\frac{\vp\big((\delta X)_{s_0t_0} U (\delta X)_{s_0t_0} \big| \cb_{s_0}\big)}{|t_0-s_0|}=\frac{\vp\big((\delta X)_{s_1t_1}U (\delta X)_{s_1t_1}\big| \cb_{s_1}\big)}{|t_1-s_1|} \ .$$
\end{lemma}

\begin{definition}\label{def:second-quanti}
We call \emph{second quantization} of $X$ the operator $\Gamma_q: \cup_{t\geq 0} \ca_t \to \ca$ defined for all $s\geq 0$ and $U\in \ca_s$ by the formula
$$\Gamma_q(U):=\vp\big((\delta X)_{s,s+1} U (\delta X)_{s,s+1} \big| \cb_{s}\big) \ .$$
In particular, for all $s\geq 0$ and $U\in \ca_s$, $\Gamma_q(U)\in \cb_s$, $\Gamma_q(U)^\ast=\Gamma_q(U^\ast)$ and 
\begin{equation}\label{contrac-property}
\|\Gamma_q(U)\|_{L^2(\vp)}\leq \|(\delta X)_{s,s+1} U (\delta X)_{s,s+1}\|_{L^2(\vp)}\leq \|X_1\|^2\|U\| \ .
\end{equation}
\end{definition}

\smallskip

\begin{remark}
For $q=0$, it is easy to check that, thanks to the freeness properties of $X$, the second quantization reduces to $\Gamma_0(U)=\vp(U)$, while in the commutative situation, that is when $q\to 1$, one has (at least morally) $\Gamma_1(U)=U$.
\end{remark}

\smallskip

In fact, we will essentially use the operator $\Gamma_q$ through the following result, which offers a quite general tool to study Itô/Stratonovich correction terms (for the sake of clarity, we have postponed the proof of this proposition to Section \ref{subsec:proof}):

\begin{proposition}\label{prop:second-quant-lim-sum}
For every adapted triprocess $\mathcal{U}\in \cac_1^\varepsilon([s,t];\ca^{\hat{\otimes}3})$ ($s<t$, $\varepsilon >0$) and every subdivision $\Delta$ of $[s,t]$ whose mesh $|\Delta|$ tends to $0$, it holds that 
\begin{equation}\label{converg-sec-quant}
\sum_{(t_i)\in \Delta} (\delta X_{t_i t_{i+1}} \sharp \mathcal{U}_{t_i})\sharp \delta X_{t_i t_{i+1}}  \longrightarrow \int_s^t \big[\id \times \Gamma_q\times \id\big](\mathcal{U}_{u})\, \mathrm{d}u \quad \text{in} \ L^2(\vp) \ ,
\end{equation}
where $\id \times \Gamma_q \times \id$ stands for the continuous extension, as an operator from $\cup_{u\geq 0}\ca_u^{\hat{\otimes}3}$ to $L^2(\vp)$, of the operator
$$(\id \times \Gamma_q \times \id)(U_1 \otimes U_2 \otimes U_3):=U_1 \Gamma_q(U_2) U_3 \quad , \ U_1,U_2,U_2\in \ca_u \ .$$
\end{proposition}

\subsection{Non-commutative Itô integral}\label{subsec:appli-free}

Let us here slightly rephrase the results of \cite{donati} regarding Itô's approach to stochastic integration with respect to $X$. 

\begin{definition}\label{defi:ito-integral}
Fix an interval $I\subset \R$. An adapted biprocess $\bu: I\to \ca \hat{\otimes} \ca$ is said to be \emph{Itô integrable} against $X$ if it is adapted and if for every partition $\Delta$ of $I$ whose mesh $|\Delta|$ tends to $0$, the sequence of Riemann sums
$$S^\Delta_X(\bu):=\sum_{t_i \in \Delta} \bu_{t_i}\sharp \delta X_{t_it_{i+1}}$$
converges in $L^2(\vp)$ (as $|\Delta |\to 0$). In this case, we call the limit of $S^\Delta_X(\bu)$ the product Itô integral of $\bu$ against $X$, and we denote it by
$$\int_I \bu_s \sharp \mathrm{d}X_s \in L^2(\vp) \ .$$
\end{definition}

Given a biprocess $\bu :I\to \ca\hat{\otimes}\ca$ and a partition $\Delta$ of $I$, we denote by $\bu^\Delta$ the step-approximation
$$\bu^\Delta :=\sum_{t_i\in \Delta} \bu_{t_i} \1_{[t_i,t_{i+1}[} \ .$$
The following isometry property, to be compared with the classical Brownian Itô isometry, is the key ingredient to identify Itô-integrable processes:

\begin{proposition}
\cite[Proposition 3.3]{donati}. For every interval $I\subset \R$, all adapted biprocesses 
$$\bu:I \to \ca \otimes \ca \quad , \quad \bv:I\to \ca \otimes \ca \ ,$$
 and all partitions $\Delta_1,\Delta_2$ of $I$, it holds that 
\begin{equation}\label{isom-ito-int}
\langle S^{\Delta_1}_X(\bu) , S^{\Delta_2}_X(\bv) \rangle_{L^2(\vp)} =\int_0^\infty \langle\langle \bu^{\Delta_1}_u,\bv^{\Delta_2}_u \rangle\rangle_q\, \mathrm{d}u \  ,
\end{equation} 
where $\langle\langle .,. \rangle\rangle_q$ is the bilinear extension of the application defined for all $U_1,U_2,V_1,V_2\in \cup_{t\geq 0}\ca_t$ as
$$\langle\langle U_1 \otimes U_2,V_1\otimes V_2 \rangle\rangle_q:=\vp\big(U_1 \Gamma_q(U_2 V_{2}^\ast) V_1^{\ast}  \big) \ .$$
\end{proposition}

\begin{corollary}\label{crit-ito-int}
Let $\bu:I\to \ca \hat{\otimes}\ca$ be an adapted biprocess such that
$$\int_I \| \bu_u\|_{\ca \hat{\otimes} \ca}^2 \, \mathrm{d}u <\infty \quad \text{and} \quad \int_I \|\bu^\Delta_u-\bu_u\|_{\ca \hat{\otimes}\ca}^2 \, \mathrm{d}u \to 0 \ \ \text{as}\  |\Delta|\to 0 \ ,$$
for every partition $\Delta$ of $I$.
Then $\bu$ is Itô integrable against $X$ and
\begin{equation}\label{ito-iso}
\Big\| \int_I \bu_u \sharp \mathrm{d}X_u \Big\|_{L^2(\vp)}^2=\int_I \langle \langle \bu_u, \bu_u \rangle\rangle_q \, \mathrm{d}u \ .
\end{equation}
\end{corollary}

\begin{proof}
Let us just provide a few details, the procedure being essential standard. Consider a sequence $\bu^n:I\to \ca \otimes \ca$ of adapted biprocesses such that for every $t\in I$, $\|\bu^n_t-\bu_t\| \to 0$. Then, given two partitions $\Delta_1,\Delta_2$ of $I$, one has by (\ref{isom-ito-int})
$$\big\| S^{\Delta_1}_X(\bu^n)-S^{\Delta_2}_X(\bu^n) \big\|_{L^2(\vp)}^2 =\int_I \langle \langle \bu^{n,\Delta_1}_u-\bu^{n,\Delta_2}_u, \bu^{n,\Delta_1}_u-\bu^{n,\Delta_2}_u \rangle\rangle_q \, \mathrm{d}u \ .$$
By applying Cauchy-Schwarz inequality and then (\ref{contrac-property}), it is readily checked that for all $\bv\in \ca_s\otimes \ca_s$, 
$$\langle\langle \bv,\bv\rangle\rangle_q \leq \|X_1\|^2 \|\bv \|^2 \ ,$$
 and so
$$\big\| S^{\Delta_1}_X(\bu^n)-S^{\Delta_2}_X(\bu^n) \big\|_{L^2(\vp)}^2 \leq \|X_1\|^2 \int_I \big\|\bu^{n,\Delta_1}_u-\bu^{n,\Delta_2}_u\big\|^2 \, \mathrm{d}u \ ,$$
which, by letting $n$ tend to infinity, leads us to
$$\big\| S^{\Delta_1}_X(\bu)-S^{\Delta_2}_X(\bu) \big\|_{L^2(\vp)}^2 \leq \|X_1\|^2 \int_I \big\|\bu^{\Delta_1}_u-\bu^{\Delta_2}_u\big\|^2 \, \mathrm{d}u \ .$$
The conclusion easily follows.
\end{proof}

\subsection{Comparison with the rough integral}
We now have all the tools to identify, as elements in $L^2(\vp)$, the rough constructions arising from Sections \ref{sec:general-rp} and \ref{sec:levy-area}. 

Let us first consider the situation at the level of the product Lévy area provided by Theorem \ref{theo:exi-levy-area-qbm}. To this end, given $0\leq s<t$ and $\bu\in \ca_s \hat{\otimes} \ca_s$, observe that, by Corollary \ref{crit-ito-int}, the biprocess $\bv_u:=(\bu \sharp \der X_{su}) \otimes 1$ is known to be Itô-integrable on $[s,t]$, which allows us to consider the integral
$$\int_s^t (\mathbf{U}\sharp \der X_{su}) \, \mathrm{d}X_u \ \in L^2(\vp)\  .$$

\begin{proposition}\label{prop:correction-ito-strato-areas}
For all $0\leq s<t$ and every $\mathbf{U}\in  \ca_s \hat{\otimes} \ca_s$, it holds that
\begin{equation}\label{correction-ito-strato-areas}
\mathbb{X}^{2,S}_{st}[\mathbf{U}]=\int_s^t (\mathbf{U}\sharp \der X_{su}) \, \mathrm{d}X_u+\frac12 (t-s)\big( \id \times \Gamma_q\big) [\bu] \quad  \text{in}\ L^2(\vp) \ ,
\end{equation}
where $\text{Id} \times \Gamma_q$ stands for the continuous extension, as an operator from $\ca_s \hat{\otimes} \ca_s$ to $L^2(\vp)$, of the operator 
$$\big(\text{Id} \times \Gamma_q\big)[U\otimes V]:=U \Gamma_q(V) \ .$$
\end{proposition}

\begin{proof}
Fix $s<t$, $\mathbf{U}\in \ca_s \hat{\otimes} \ca_s$, and let $\Dti^n$ be the subdivision obtained by adding the two times $s,t$ to the dyadic partition $D^n:=\{i /2^{n},\, i\geq 0\}$. Denote by $\Xti^n$ the linear interpolation of $X$ along $\Dti^n$ and set $\Xha^n:=\sum_{\tha_i} X_{\tha_i} 1_{[\tha_i,\tha_{i+1})}$ where $\{s=\tha_1 < \ldots < \tha_n=t\}:=\Dti^n \cap [s,t]$. Besides, we recall that the notation $\mathbb{X}^{2,D^n}$ (or $\mathbb{X}^{2,\Dti^n}$) has been introduced in (\ref{levy-area-appr}).

\smallskip

\noindent
Using only the $1/2$-Hölder regularity of $X$ (see (\ref{q-bm-holder})), it is easy to check that for every $\mathbf{U}\in \ca_s \hat{\otimes} \ca_s$,
\begin{equation}\label{comp-d-dti}
\| \mathbb{X}^{2,D^n}_{st}[\mathbf{U}]-\mathbb{X}^{2,\Dti^n}_{st}[\mathbf{U}] \| \leq c \|X_1\|^2 \|\mathbf{U} \| \lln t-s \rrn^{2\ga} 2^{-n(1/2-\ga)} \ ,
\end{equation}
for some universal constant $c$ and for every $\ga\in (0,1/2)$. Thus, by Theorem \ref{theo:exi-levy-area-qbm}, we can assert that $\mathbb{X}^{2,\Dti^n}_{st}[\mathbf{U}]$ converges to $\mathbb{X}^{2,S}_{st}[\mathbf{U}]$ for the operator norm (and accordingly in $L^2(\vp)$).
Now write
\begin{eqnarray}
\mathbb{X}^{2,\Dti^n}_{st}[\mathbf{U}]&=&\int_s^t (\bu \sharp \delta \Xti^n_{su} )\, \mathrm{d}\Xti^n_u\\
&=& \sum_{k=1}^{n-1} \frac{1}{\tha_{k+1}-\tha_k} \int_{\tha_k}^{\tha_{k+1}}\bu \sharp\big( \delta X_{s\tha_k}+\frac{u-\tha_k}{\tha_{k+1}-\tha_k}(\der X)_{\tha_{k}\tha_{k+1}} \big) \, \mathrm{d}u \, (\der X)_{\tha_{k}\tha_{k+1}}\nonumber\\
&=& \sum_{k=1}^{n-1} (\bu \sharp \delta X_{s \tha_k} )\, \der X_{\tha_{k}\tha_{k+1}}+\frac{1}{2}\sum_{k=1}^{n-1} \bu \sharp (\der X)_{\tha_{k}\tha_{k+1}} \, (\der X)_{\tha_{k}\tha_{k+1}}\nonumber\\
&=& \int_s^t (\bu \sharp \delta \Xha^n_{su})  \, \mathrm{d}X_u+\frac{1}{2}\sum_{k=1}^{n-1} \bu \sharp (\der X)_{\tha_{k}\tha_{k+1}} \, (\der X)_{\tha_{k}\tha_{k+1}}.\label{two-terms}
\end{eqnarray}
Thanks to (\ref{ito-iso}), it holds that
\begin{eqnarray*}
\lefteqn{\big\| \int_s^t (\bu \sharp \delta \Xha^n_{su})  \, \mathrm{d}X_u-\int_s^t (\bu \sharp \delta X_{su})  \, \mathrm{d}X_u \big\|_{L^2(\vp)}^2}\\
&=&\int_s^t \langle\langle\, (\bu \sharp[ \delta \Xha^n_{su}-\delta X_{su}])\otimes 1,(\bu \sharp [ \delta \Xha^n_{su}- \delta X_{su}])\otimes 1\, \rangle\rangle_q \, \mathrm{d}u  \\
 &=& \int_s^t \big\| \bu \sharp[ \delta \Xha^n_{su}-\delta X_{su}] \big\|_{L^2(\vp)}^2 \, \mathrm{d}u \\
&\leq &  \|\bu\|^2  \sum_{k=1}^{n-1} \int_{\tha_k}^{\tha_{k+1}} \|X_{\tha_k}-X_u\|^2 \, \mathrm{d}u \\
& \leq & \|\bu\|^2 \sum_{k=1}^{n-1} \int_{\tha_k}^{\tha_{k+1}} (u-\tha_k) \, \mathrm{d}u \ \leq \ \frac12 \|\bu\|^2 2^{-n} \lln t-s \rrn \ \to \ 0 \ .
\end{eqnarray*}
Observe finally that the limit of the second term in (\ref{two-terms}) is immediately provided by Proposition \ref{prop:second-quant-lim-sum}, which achieves the proof of (\ref{correction-ito-strato-areas}).
\end{proof}

\smallskip

Let us now extend the correction formula (\ref{correction-ito-strato-areas}) to any adapted controlled biprocess, that is to the class of biprocesses introduced in Definition \ref{def:control-proc}. Using again Corollary \ref{crit-ito-int}, it is easy to check that, as an adapted Hölder path in $\ca \hat{\otimes} \ca$, any such controlled biprocess is Itô-integrable when considered on an interval $I$ of finite Lebesgue measure. This puts us in a position to state the formula:

\begin{corollary}\label{prop:transition} 
For all $0\leq s<t$ and every adapted controlled biprocess $\bu\in \mathbf{Q}_X([s,t])$ with decomposition (\ref{decompo-bipro}), it holds that
\begin{equation}\label{transi-ito-strato}
 \int_s^t \bu_u\sharp \mathrm{d}\mathbb{X}^{S}_u=\int_s^t \bu_u \sharp \mathrm{d}X_u
+\frac{1}{2} \int_s^t (\id \times \Gamma_q \times \id)[\mathcal{U}^{X,1}_u+\mathcal{U}^{X,2}_u] \, \mathrm{d}u \quad \text{in} \ L^2(\vp) \ .
\end{equation}
\end{corollary}

\begin{proof}
The transition from (\ref{correction-ito-strato-areas}) to (\ref{transi-ito-strato}) follows from the very same Taylor-expansion argument as in the proof of \cite[Proposition 5.6]{deya-schott} (related to the free case), and so, for the sake of conciseness, we do not repeat it here. 
\end{proof}

At this point, observe that the combination of Proposition \ref{pro-approx-1} and Corollary \ref{prop:transition} immediately yields the following $q$-extension of Itô/Stratonovich formula: for all $f\in\mathbf{F}_3$ and $s<t$, 
$$\der( f(X))_{st}=\int_s^t \partial f(X_u) \sharp \mathrm{d}\mathbb{X}^S_u=\int_s^t \partial f(X_u) \sharp \mathrm{d}X_u+\int_s^t [\id \times \Gamma_q \times \id](\partial^2 f(X_u)) \, \mathrm{d}u \ .$$

\smallskip

As another spin-off of Formula (\ref{transi-ito-strato}), we can finally derive an expression of the rough Stratonovich integral $\int_s^t \bu_u\sharp \mathrm{d}\mathbb{X}^S_u$ as the $L^2(\vp)$-limit of \enquote{mean-value} Riemann sums. The result, which emphasizes the analogy between the rough construction and the classical (commutative) Stratonovich integral, can be stated as follows:  
\begin{corollary}\label{coro:identif-strato}
For all $0\leq s<t$ and every adapted controlled biprocess $\bu\in \mathbf{Q}_X([s,t])$, it holds that
\begin{equation}\label{mean-val-sum}
\int_s^t \bu_u\sharp \mathrm{d}\mathbb{X}^S_u = \lim_{|\Delta|\to 0} \sum_{(t_i)\in \Delta} \frac{1}{2}\big(\bu_{t_i}+\bu_{t_{i+1}}\big) \sharp \delta X_{t_it_{i+1}} \quad \text{in} \ L^2(\vp) \ ,
\end{equation}
for any subdivision $\Delta$ of $[s,t]$ whose mesh $|\Delta|$ tends to $0$.
\end{corollary}
\begin{proof}
For any subdivision $\Delta=(t_i)$ of $[s,t]$, write
\begin{eqnarray*}
\lefteqn{\frac{1}{2}\big(\bu_{t_i}+\bu_{t_{i+1}}\big)\sharp \delta X_{t_it_{i+1}}}\\
 &=&\bu_{t_i}\sharp \delta X_{t_it_{i+1}}+\frac{1}{2}\delta\bu_{t_it_{i+1}}\sharp \delta X_{t_it_{i+1}}\\
&=&\bu_{t_i}\sharp \delta X_{t_it_{i+1}}\\
& &+\frac12 \big[(\der X_{t_i t_{i+1}} \sharp \mathcal{U}_{t_i}^{X,1})\sharp \delta X_{t_it_{i+1}} +(\mathcal{U}_{t_i}^{X,2} \sharp \der X_{t_it_{i+1}})\sharp \delta X_{t_it_{i+1}}+ \bu^\flat_{t_it_{i+1}}\sharp \delta X_{t_it_{i+1}}\big]\ ,
\end{eqnarray*}
and observe that, with the notations of Section \ref{subsec:gubi}, we have 
$$\|\bu^\flat_{t_it_{i+1}}\sharp \delta X_{t_it_{i+1}}\| \leq |t_{i+1}-t_i|^{2\ga+1/2}\|X_1\|\, \cn[\bu^\flat;\cac_2^{2\ga}([s,t])]\ .$$
Taking the sum over $i$ and then letting $|\Delta|$ tend to $0$, we get by Proposition \ref{prop:second-quant-lim-sum} that the sum in (\ref{mean-val-sum}) converges in $L^2(\vp)$ to the right-hand side of (\ref{transi-ito-strato}), which leads us to the conclusion.
\end{proof}

\

\subsection{Proof of Proposition \ref{prop:second-quant-lim-sum}}\label{subsec:proof}
When $\mathcal{U}_t=U_t \otimes V_t \otimes W_t$, the convergence property (\ref{converg-sec-quant}) has been shown in the proof of \cite[Theorem 3.2]{donati}. However, since we want the formula to hold for general adapted triprocesses here, we need to exhibit additional controls. Let $\mathcal{U}^n:[s,t]\to \ca^{\otimes 3}$ be a sequence of adapted triprocesses such that $\big\| \mathcal{U}^n_u-\mathcal{U}_u\big\|\to 0$ for every $u\in [s,t]$, and fix a subdivision $\Delta=(t_i)$ of $[s,t]$. Then set successively $Y_i:=\delta X_{t_it_{i+1}}$,
$$
S_{\Delta}(\mathcal{U}):=\sum_{(t_i)\in \Delta}\big\{ ( Y_i\sharp \mathcal{U}_{t_i}) \sharp Y_i - (t_{i+1}-t_i)\big[\id \times \Gamma_q\times \id\big](\mathcal{U}_{t_i})\big\} 
$$
$$
\text{and}\quad S_{\Delta}^n(\mathcal{U}):=\sum_{(t_i)\in \Delta}\big\{ ( Y_i\sharp \mathcal{U}^n_{t_i}) \sharp Y_i- (t_{i+1}-t_i)\big[\id \times \Gamma_q\times \id\big](\mathcal{U}^n_{t_i})\big\} \ .
$$
If $\mathcal{U}^n_t:=\sum_{\ell \leq L^n_t} U^n_{t,\ell} \otimes V^n_{t,\ell} \otimes W^n_{t,\ell} \in \ca_t^{\otimes 3}$, $S_{\Delta}^n(\mathcal{U})$ thus corresponds to 
$$S_{\Delta}^n(\mathcal{U})=\sum_{(t_i)\in \Delta}\sum_{\ell \leq L^n_{t_i}} M^n_{i,\ell} \ ,$$
with
$$M^n_{i,\ell}:=U^n_{t_i,\ell}Y_i V^n_{t_i,\ell} Y_i W^n_{t_i,\ell}-(t_{i+1}-t_i)U^n_{t_i,\ell}\Gamma_q(V^n_{t_i,\ell})W^n_{t_i,\ell} \ ,$$
so
\begin{equation}\label{decomp-nor-so-bis}
\|S_\Delta^n(\mathcal{U})\|_{L^2(\vp)}^2 = \sum_{(t_{i_1})\in \Delta}\sum_{(t_{i_2})\in \Delta}\sum_{\ell_1 \leq L^n_{t_{i_1}}}\sum_{\ell_2 \leq L^n_{t_{i_2}}}\vp\big( M^n_{i_1,\ell_1}(M^n_{i_2,\ell_2})^\ast  \big) \ .
\end{equation}
For more clarity, let us set $U^n_{i,\ell}:=U^n_{t_i,\ell}$, $V^n_{i,\ell}:=V^n_{t_i,\ell}$, $W^n_{i,\ell}:=W^n_{t_i,\ell}$, and consider then the expansion
\begin{eqnarray}
\lefteqn{\vp\big( M^n_{i_1,\ell_1}(M^n_{i_2,\ell_2})^\ast  \big)}\nonumber\\
&=& \vp\big(U^n_{i_1,\ell_1}Y_{i_1} V^n_{i_1,\ell_1}Y_{i_1}W^n_{i_1,\ell_1}W^{n,\ast}_{i_2,\ell_2}Y_{i_2} V^{n,\ast}_{i_2,\ell_2}Y_{i_2}U^{n,\ast}_{i_2,\ell_2}\big)\nonumber\\
& &-(t_{i_2+1}-t_{i_2}) \vp\big(U^n_{i_1,\ell_1}Y_{i_1} V^n_{i_1,\ell_1}Y_{i_1}W^n_{i_1,\ell_1}W^{n,\ast}_{i_2,\ell_2}\Gamma_q(V^{n,\ast}_{i_2,\ell_2})U^{n,\ast}_{i_2,\ell_2}\big)\nonumber\\
& &-(t_{i_1+1}-t_{i_1})\vp\big( U^n_{i_1,\ell_1}\Gamma_q(V^n_{i_1,\ell_1})W^n_{i_1,\ell_1}W^{n,\ast}_{i_2,\ell_2}Y_{i_2} V^{n,\ast}_{i_2,\ell_2}Y_{i_2}U^{n,\ast}_{i_2,\ell_2} \big)\nonumber\\
& &+(t_{i_1+1}-t_{i_1})(t_{i_2+1}-t_{i_2}) \vp\big(U^n_{i_1,\ell_1}\Gamma_q(V^n_{i_1,\ell_1})W^n_{i_1,\ell_1} W^{n,\ast}_{i_2,\ell_2}\Gamma_q(V^{n,\ast}_{i_2,\ell_2})U^{n,\ast}_{i_2,\ell_2} \big) \ .\label{mai-ter-bis}
\end{eqnarray}

\

\noindent
\emph{Step 1: Non-diagonal terms ($i_1\neq i_2$).} Observe first that if for instance $i_1 < i_2$, we have, by combining Lemma \ref{lem:second-quanti} and Definition \ref{def:second-quanti}, 
\begin{eqnarray*}
\lefteqn{\vp\big(U^n_{i_1,\ell_1}Y_{i_1} V^n_{i_1,\ell_1}Y_{i_1}W^n_{i_1,\ell_1}W^{n,\ast}_{i_2,\ell_2}Y_{i_2} V^{n,\ast}_{i_2,\ell_2}Y_{i_2}U^{n,\ast}_{i_2,\ell_2}\big)}\\
&=&\vp\big(U^n_{i_1,\ell_1}Y_{i_1} V^n_{i_1,\ell_1}Y_{i_1}W^n_{i_1,\ell_1}W^{n,\ast}_{i_2,\ell_2}\vp\big( Y_{i_2} V^{n,\ast}_{i_2,\ell_2}Y_{i_2}\big| \cb_{t_{i_2}}\big) U^{n,\ast}_{i_2,\ell_2}\big)\\
&=&(t_{i_2+1}-t_{i_2}) \vp\big(U^n_{i_1,\ell_1}Y_{i_1} V^n_{i_1,\ell_1}Y_{i_1}W^n_{i_1,\ell_1}W^{n,\ast}_{i_2,\ell_2}\Gamma_q(V^{n,\ast}_{i_2,\ell_2}) U^{n,\ast}_{i_2,\ell_2}\big) \ ,
\end{eqnarray*}
and with the same conditioning argument
\begin{align*}
&\vp\big( U^n_{i_1,\ell_1}\Gamma_q(V^n_{i_1,\ell_1})W^n_{i_1,\ell_1}W^{n,\ast}_{i_2,\ell_2}Y_{i_2} V^{n,\ast}_{i_2,\ell_2}Y_{i_2}U^{n,\ast}_{i_2,\ell_2} \big) \\
&=(t_{i_2+1}-t_{i_2}) \vp\big(U^n_{i_1,\ell_1}\Gamma_q(V^n_{i_1,\ell_1})W^n_{i_1,\ell_1} W^{n,\ast}_{i_2,\ell_2}\Gamma_q(V^{n,\ast}_{i_2,\ell_2})U^{n,\ast}_{i_2,\ell_2} \big) \ ,
\end{align*}
so that, going back to (\ref{mai-ter-bis}), one has $\vp\big( M^n_{i_1,\ell_1}(M^n_{i_2,\ell_2})^\ast  \big)=0$. Similar arguments lead to the same conclusion when $i_2<i_1$.

\smallskip

\noindent
\emph{Step 2: Diagonal terms ($i_1=i_2=i$).} First, observe that with the same conditioning argument as above, decomposition (\ref{mai-ter-bis}) actually reduces to
\begin{eqnarray*}
\vp\big( M^n_{i,\ell_1}(M^n_{i,\ell_2})^\ast  \big)&=&\vp\big(U^n_{i,\ell_1}Y_{i} V^n_{i,\ell_1}Y_{i}W^n_{i,\ell_1}W^{n,\ast}_{i,\ell_2}Y_{i} V^{n,\ast}_{i,\ell_2}Y_{i}U^{n,\ast}_{i,\ell_2}\big)\\
& &-(t_{i+1}-t_{i})^2 \vp\big(U^n_{i,\ell_1}\Gamma_q(V^n_{i,\ell_1})W^n_{i,\ell_1} W^{n,\ast}_{i,\ell_2}\Gamma_q(V^{n,\ast}_{i,\ell_2})U^{n,\ast}_{i,\ell_2} \big) \ .
\end{eqnarray*} 
Now, on the one hand, using (\ref{q-bm-holder}) and the Cauchy-Schwarz inequality,
\begin{eqnarray*}
\lefteqn{\big|\vp\big(U^n_{i,\ell_1}Y_{i} V^n_{i,\ell_1}Y_{i}W^n_{i,\ell_1}W^{n,\ast}_{i,\ell_2}Y_{i} V^{n,\ast}_{i,\ell_2}Y_{i}U^{n,\ast}_{i,\ell_2}\big)\big|}\\
 &\leq& \|Y_i\|^4 \| U^n_{i,\ell_1}\| \| V^n_{i,\ell_1}\| \|W^n_{i,\ell_1}\| \| W^{n}_{i,\ell_2}\| \|V^{n}_{i,\ell_2}\| \|U^{n}_{i,\ell_2}\|\\
&\leq& (t_{i+1}-t_i)^2\|X_1\|^4 \| U^n_{i,\ell_1}\| \| V^n_{i,\ell_1}\| \|W^n_{i,\ell_1}\| \| W^{n}_{i,\ell_2}\| \|V^{n}_{i,\ell_2}\| \|U^{n}_{i,\ell_2}\| \ .
\end{eqnarray*}
On the other hand, using the definition of $\Gamma_q(V_{j_1})$,
\begin{eqnarray*}
\lefteqn{\big| \vp\big(U^n_{i,\ell_1}\Gamma_q(V^n_{i,\ell_1})W^n_{i,\ell_1} W^{n,\ast}_{i,\ell_2}\Gamma_q(V^{n,\ast}_{i,\ell_2})U^{n,\ast}_{i,\ell_2} \big) \big|}\\
 &=&\big| \vp\big(U^n_{i,\ell_1}(\delta X)_{t_i, t_i+1}V^n_{i,\ell_1}(\delta X)_{t_i,t_i+1}W^n_{i,\ell_1} W^{n,\ast}_{i,\ell_2}\Gamma_q(V^{n,\ast}_{i,\ell_2})U^{n,\ast}_{i,\ell_2} \big) \big|\\
&\leq& \big\|U^{n,\ast}_{i,\ell_2} U^n_{i,\ell_1}(\delta X)_{t_i, t_i+1}V^n_{i,\ell_1}(\delta X)_{t_i,t_i+1}W^n_{i,\ell_1} W^{n,\ast}_{i,\ell_2} \big\|_{L^2(\vp)} \big\| \Gamma_q(V^{n,\ast}_{i,\ell_2}) \big\|_{L^2(\vp)} \ ,
\end{eqnarray*}
which, by (\ref{contrac-property}), entails that
\begin{eqnarray*}
\lefteqn{\big| \vp\big(U^n_{i,\ell_1}\Gamma_q(V^n_{i,\ell_1})W^n_{i,\ell_1} W^{n,\ast}_{i,\ell_2}\Gamma_q(V^{n,\ast}_{i,\ell_2})U^{n,\ast}_{i,\ell_2} \big) \big|}\\
&\leq & \|X_1\|^4 \| U^n_{i,\ell_1}\| \| V^n_{i,\ell_1}\| \|W^n_{i,\ell_1}\| \| W^{n}_{i,\ell_2}\| \|V^{n}_{i,\ell_2}\| \|U^{n}_{i,\ell_2}\| \ .
\end{eqnarray*}

\

\noindent
Going back to (\ref{decomp-nor-so-bis}), we have thus shown that 
$$
\|S_\Delta^n(\mathcal{U})\|_{L^2(\vp)}^2\leq \|X_1\|^4 \sum_{(t_{i})\in \Delta}(t_{i+1}-t_i)^2\Big(\sum_{\ell \leq L^n_{t_{i}}}\| U^n_{i,\ell}\| \| V^n_{i,\ell}\| \|W^n_{i,\ell}\|\Big)^2 \ ,\\
$$
and so we can assert that
\begin{eqnarray*}
\|S_\Delta^n(\mathcal{U})\|_{L^2(\vp)}^2&\leq& \|X_1\|^4 \sum_{(t_{i})\in \Delta}(t_{i+1}-t_i)^2 \| \mathcal{U}^n_{t_i}\|^2 \\
&\leq&2 \|X_1\|^4\bigg\{ \sum_{(t_{i})\in \Delta}(t_{i+1}-t_i)^2 \| \mathcal{U}^n_{t_i}-\mathcal{U}_{t_i}\|^2+\big(\sup_{u\in [s,t]}\| \mathcal{U}_{u}\|^2\big) |t-s| |\Delta| \bigg\} \ .
\end{eqnarray*}
By letting $n$ tend to infinity first, we can conclude that $\|S_{\Delta}(\mathcal{U})\|_{L^2(\vp)}^2 \to \ 0$ as the mesh $|\Delta|$ tends to $0$. The convergence
$$\sum_{(t_i)\in \Delta}(t_{i+1}-t_i)\big[\id \times \Gamma_q\times \id\big](\mathcal{U}_{t_i}) \to \int_s^t \big[\id \times \Gamma_q\times \id\big](\mathcal{U}_{u}) \, \mathrm{d}u \quad \text{in} \ L^2(\vp)$$
follows easily from the regularity of $\mathcal{U}$, by noting that for every $u$ and every $\mathcal{V}\in \ca_u^{\hat{\otimes}3}$, 
$$\big\| \big[\id \times \Gamma_q\times \id\big](\mathcal{V}) \big\|_{L^2(\vp)} \leq \| \mathcal{V}\| \ .$$
 This achieves the proof of our statement.

\end{document}